\newcommand{\siamyesno}[2]{#2}   
\siamyesno{
\documentclass[onetabnum,onefignum,nohypdvips,final]{siamart171218}
\usepackage{amssymb,amsmath,epsfig,verbatim,enumitem}
\newtheorem{remark}[theorem]{Remark}
\newtheorem{ass}[theorem]{Assumption}
}{
\documentclass[11pt,a4paper]{article}
\usepackage{amssymb,amsmath,amsthm,epsfig,verbatim,xcolor,enumitem}
\newtheorem{theorem}{\sc Theorem.}[section]
\newtheorem{lemma}[theorem]{\sc Lemma.}
\newtheorem{remark}[theorem]{\sc Remark.}

\renewcommand{\theequation}{\arabic{section}.\arabic{equation}}
\newenvironment{AMS}%
{{\upshape\bfseries AMS subject classifications. }\ignorespaces}{}
\newenvironment{keywords}{{\upshape\bfseries Key words. }\ignorespaces}{}

}
\usepackage{ifpdf}  
\ifpdf
  \DeclareGraphicsExtensions{.eps,.pdf,.png,.jpg}
\else
  \DeclareGraphicsExtensions{.eps,.ps}
\fi
\newcommand{\bRplus}{{\bR}_{>0}}

\newcommand{\RZ}{{\bR} \slash {\mathbb Z}}
\newcommand{\bR}{{\mathbb R}}
\newcommand{\bB}{{\mathbb B}}

\newcommand{\bN}{{\mathbb N}}

\newcommand{\drho}{\;{\rm d}\rho}

\newcommand{\Id}{I\!d}
\newcommand{\dd}[1]{\frac{\rm d}{{\rm d}#1}}
\newcommand{\ddt}{\dd{t}}
\newcommand{\ratio}{{\mathfrak r}}
\newcommand{\Vh}{\underline{V}^h}
\newcommand{\Vpartial}{\underline{V}_\partial}
\newcommand{\Vhpartial}{\underline{V}_\partial^h}

\def\epsilon{\varepsilon} 
\def\hat{\widehat}
\def\widebar{\overline}
\def\bar{\widebar}
\def\tilde{\widetilde}

\siamyesno{}{
\textwidth 455pt \oddsidemargin 0pt \evensidemargin 0pt \headsep
0pt \headheight 0pt \textheight 655pt \parskip 10pt \parindent 0pt

}

\begin{document}
\title{
A filtered time stepping scheme for curve shortening flow \\ 
for open and closed curves}
\author{Klaus Deckelnick\footnotemark[2]\ \and 
        Robert N\"urnberg\footnotemark[3]}

\renewcommand{\thefootnote}{\fnsymbol{footnote}}
\footnotetext[2]{Institut f\"ur Analysis und Numerik,
Otto-von-Guericke-Universit\"at Magdeburg, 39106 Magdeburg, Germany \\
{\tt klaus.deckelnick@ovgu.de}}
\footnotetext[3]{Dipartimento di Matematica, Universit\`a di Trento,
38123 Trento, Italy \\ {\tt robert.nurnberg@unitn.it}}

\date{}

\maketitle

\begin{abstract}
We propose a filtered time stepping finite element scheme for 
curve shortening flow of open and closed curves in arbitrary codimension that is second-order
accurate in time. Open curves are assumed to
evolve inside a given domain $\Omega \subset \bR^n$, $n\geq2$,
and meet the external boundary $\partial\Omega$ orthogonally.
We prove optimal error bounds for the $L^2$-- and $H^1$--norms.
In practice only a single linear system needs to be solved at each time step. 
Numerical experiments confirm the accuracy and practicality of the 
introduced method, including an asymptotic equidistribution property.
\end{abstract} 



\renewcommand{\thefootnote}{\arabic{footnote}}

\siamyesno{
\pagestyle{myheadings}
\thispagestyle{plain}
\markboth{K. DECKELNICK AND R. N\"URNBERG}
{}
}{}

\setcounter{equation}{0}
\section{Introduction} 

We consider the evolution of a curve inside a domain $\Omega \subset \bR^n$, 
$n \geq 2$,
with normal velocity given by its curvature. In the case of an open curve,
it meets the boundary $\partial\Omega$ at a right angle. It is
well known that this geometric evolution law is the $L^2$--gradient flow of the
curve's length, and is therefore known as curve shortening flow.

Following \cite{DeckelnickE98}, 
we consider a parametric approach and aim to find a mapping 
$x: \overline I \times [0,T] \to \bR^n$ such that
\begin{subequations} \label{eq:pde}
\begin{equation}  \label{eq:csf}
{|x_\rho|^2} x_t - {x_{\rho \rho}}   = 0\qquad
\text{ in } I \times (0,T],
\end{equation}
where $I = \RZ$ in the case of a
closed curve and $I=(0,1)$ for an open curve. In the latter case, we prescribe
the boundary conditions
\begin{equation}  \label{eq:bc} 
x(\rho,t)  \in \partial \Omega,  \; x_\rho(\rho,t) \cdot v  =0   \mbox{ for all } v \in T_{x(\rho,t)} \partial \Omega \quad \qquad \forall (\rho,t) \in  \partial I \times (0,T].
\end{equation}
Here,  $T_z \partial \Omega$ denotes the tangent space at $z \in \partial \Omega$.
Finally, we impose the initial condition
\begin{equation} \label{eq:pded} 
x(\cdot,0)  = x_0  \text{ in } \overline {I}.
\end{equation}
\end{subequations}
We remark that the formulation \eqref{eq:csf} for curve
shortening flow can be derived with the help of
the DeTurck trick, see e.g.\ \cite{ElliottF17}. 
It can be easily shown that
solutions to \eqref{eq:pde} reduce the Dirichlet energy $\int_I |x_\rho|^2
\drho$ in time, cf.\ \eqref{eq:contstab} below, which means that the solutions
will be driven towards parameterizations that are proportional to arclength.
On the discrete level, this yields approximations with 
well distributed vertices that asymptotically become equidistributed. 
For a theoretical background on the flow, we refer to
\cite{GageH86,Grayson87} in the case of closed curves and to
\cite{RubinsteinSK89,KatsoulakisKR95,NguyenV26preprint} 
in the case of open curves. 
In particular, well-posedness of \eqref{eq:pde} for sufficiently regular
$\partial\Omega$ and $x_0$, and a sufficiently small $T>0$,
is shown in \cite{KatsoulakisKR95} for the planar case, and in 
\cite{NguyenV26preprint} for general $n\geq2$.

The last three
decades have seen a lot of interest in the parametric approximation of
curve shortening flow for closed curves, see the review articles
\cite{DeckelnickDE05,bgnreview} and the references therein. Let us mention in
particular \cite{Dziuk94,DeckelnickD95,curves3d}, in which curves evolving
in higher codimension are allowed. As the
numerical analysis for geometric evolution equations matures, the focus in
recent years has shifted towards higher order methods in space and time. Corresponding schemes,
albeit without error analysis, have been proposed in 
\cite{BalazovjechM11,MackenzieNRI19,DuanLZ21,JiangSZ24,JiangSZ24a,ZhangAF26preprint},
see \cite[Section~1]{csftime} for a more detailed description of these contributions. 
To the best of our knowledge, the first rigorous
result for \eqref{eq:pde} was obtained by the present authors in \cite{csftime}, where second-order error bounds for a
predictor-corrector scheme in the case
of closed curves were derived. Subsequently, analogous bounds were proved in
\cite{LiWW25preprint,Duan26preprint} for a Crank--Nicolson scheme and a BDF2 method. Note
that \cite{LiWW25preprint} is concerned with mean curvature flow of axisymmetric
surfaces, giving rise to an evolution equation for the profile curve that is related to
curve shortening flow. Let us also mention that \cite{BinzK23} contains an error analysis
for curve shortening flow in higher codimension in a setting that uses additional variables
and a BDF method for time discretization.

In contrast, the approximation of curve
shortening flow for open curves has been less well studied. Let us mention 
\cite{DeckelnickE98}, where error estimates for a semidiscretization of \eqref{eq:pde} in
the planar case are obtained. It is the aim of this paper to propose and analyze a fully discrete scheme, which
applies to curves evolving by \eqref{eq:pde} in any codimension, and which is second order in time.
Here, the main difficulty compared to the case of closed curves arises from estimating the boundary error terms. 
A natural strategy is to combine corresponding ideas from \cite{DeckelnickE98}
with the predictor-corrector approach from \cite{csftime} in order to obtain a method that is second order in time. In order to 
mimic the analysis of \cite{DeckelnickE98} it is necessary to write some of the boundary error terms as discrete time
derivatives, which, however, does not seem to be possible within the approach of \cite{csftime}. Instead, inspired by \cite{LiWF23}, we propose a scheme
that uses a backward Euler type step followed by a postprocessing step which simply
interpolates linearly between different discrete solutions. This procedure results in a BDF2 time discretization, which turns out to be more favourable, but
whose initialisation needs values of the discrete solution at times $t=0$ and $t=\Delta t$, with $\Delta t$ denoting the time step size. 
The definition of the discrete solution at $t=\Delta t$
requires particular care, in order to preserve the second-order accuracy of the scheme. To this end, we propose a linear system which uses curvature
information of $\partial\Omega$ for defining appropriate discrete boundary conditions. 

The remainder of the paper is organised as follows. 
In Section~\ref{sec:fd} we introduce our finite element approximation to \eqref{eq:pde} and prove its well-posedness. Our main result, 
an optimal error estimate, is stated and proved in Section~\ref{sec:ana}. For its proof it is convenient to rewrite the scheme as a finite difference
method.
In Section~\ref{sec:nr} we present some numerical simulations that confirm
the theoretical results and show the practicality of our proposed method. In the Appendix we introduce and analyse the linear system which
we use to initialise our scheme.

\subsection*{Notation}
For $\ell\in \bN_0$ we denote the norm of the Sobolev space
$H^{\ell}(I)$ by $\|\cdot\|_\ell$, with the associated semi-norm written as
$|\cdot|_\ell$. We will denote the $L^2$--inner product in $I$ by
$(\cdot,\cdot)$.
These notations naturally extend to vector functions, and we will write 
$[H^{\ell}(I)]^n$ for a vector function with $n$ components.
Throughout this paper, $c$ denotes a generic positive constant independent of 
the mesh parameter $h$ and the time step size $\Delta t$.
At times $\epsilon$ will play the role of a (small)
positive parameter, with $c_\epsilon>0$ depending on $\epsilon$, but
independent of $h$ and $\Delta t$.

\setcounter{equation}{0}
\section{Weak formulation and finite element discretisation} \label{sec:fd}

In what follows, and similarly to \cite{DeckelnickE98}, 
we assume that $\Omega \subset \bR^n$ is a domain whose boundary
$\partial\Omega$ can be described as the zero level set of a smooth function.
In particular, let $U$ be some open neighbourhood of $\partial \Omega$, and let
$F\in C^3(U)$ be such that
\begin{equation} \label{eq:defF}
\partial\Omega = \{ z \in U : F(z) = 0 \}
\qquad \mbox{and}
\qquad |\nabla\,F(z)| =1 \quad \forall z \in \partial \Omega.
\end{equation}
Let $x_0 \in H^1(I)$ with $F(x_0)=0$ on $\partial I$. 
For a mapping $x :I \times [0,T] \to \bR^n$ satisfying $x(\cdot,0)=x_0$,  the boundary conditions \eqref{eq:bc} can then be equivalently formulated in the form
\begin{subequations} \label{eq:pdebc}
\begin{align}
x_t  \cdot \nabla F(x) &=0 \qquad \text{on } \partial I \times (0,T], \label{eq:pdeb} \\
P(x) x_\rho & =0 \qquad \text{on } \partial I \times (0,T], \label{eq:pdec}
\end{align}
\end{subequations}
where $P(x) = \Id - \nabla F(x) \otimes \nabla F(x)$ is the projection
onto $T_x \partial \Omega$.
The first relation implies that $F(x(\cdot,t))=F(x_0)=0$ on $\partial I$, so that $x(\rho,t) \in \partial \Omega$ for $\rho \in \partial I$ and  $t \in (0,T]$,
while \eqref{eq:pdec} is clearly equivalent to the second condition in 
\eqref{eq:bc}. 
Next, define for $w \in [H^1(I)]^n$ the function space
\[
\Vpartial(w) = \{ \eta \in [H^1(I)]^n : \eta \cdot \nabla F(w) = 0
\text{ on } \partial I\}.
\]
A weak formulation of \eqref{eq:pde} is then given by: Find $x :I \times [0,T] \to \bR^n$ such 
that $x(\cdot,0)= x_0$, $x_t(\cdot,t) \in \Vpartial(x(t))$ for $t \in (0,T]$ and 
\begin{align} 
( x_t \cdot \eta, |x_\rho|^2 ) + ( x_\rho , \eta_\rho ) = 0
\qquad \forall \eta \in \Vpartial(x). \label{eq:csfweak}
\end{align}
It is not difficult to verify that if $x: I \times [0,T] \to \bR^n$
is a solution of \eqref{eq:csfweak} that is sufficiently regular, then
$x$ satisfies \eqref{eq:csf} and \eqref{eq:pdebc}, and hence \eqref{eq:pde}. Note that \eqref{eq:pdec} arises as the natural boundary condition from \eqref{eq:csfweak}. By choosing $\eta=x_t$ in \eqref{eq:csfweak}
we immediately see that
\begin{equation} \label{eq:contstab}
( | x_t |^2 , | x_\rho |^2 ) + \tfrac{1}{2} \ddt ( | x_\rho |^2 , 1) =0.
\end{equation}

Let us next use the weak formulation \eqref{eq:csfweak} in order to discretise our problem. We decompose $[0,1]$ into the subintervals $I_j=[\rho_{j-1},\rho_j]$, where $\rho_j=jh$, $j=0,1,\ldots,J$, $J\geq 2$.
In the case of a closed curve we identify $\rho_J = \rho_0$. Moreover, we let
$J_0 = J$ if $I = \RZ$ and $J_0 = J-1$ if $I=(0,1)$.
For two piecewise continuous functions, with possible jumps at the 
nodes $\{\rho_j\}_{j=1}^{J_0}$, we define the mass lumped $L^2$--inner product 
\begin{equation*} 
( u, v )^h = \tfrac12\sum_{j=1}^J h
\left[(u\cdot v)(\rho_j^-) + (u \cdot v)(\rho_{j-1}^+)\right],
\end{equation*}
where $(u \cdot v)(\rho_j^\pm)=\underset{\delta\searrow 0}{\lim}\ 
(u \cdot v)(\rho_j\pm\delta)$. 
We also define the finite element space
\[
V^h = \{\chi \in C^0(\overline I) : \chi\!\mid_{I_j} 
\text{ is affine},\ j=1,\ldots, J\} 
\]
as well as $\Vh = [V^h]^n$ and, for $w\in [H^1(I)]^n$,
$\Vhpartial(w) = \Vh \cap \Vpartial(w)$.
In order to discretize in time, let $t_m=m \Delta t$, $m=0,\ldots,M$, 
with the uniform time step size $\Delta t = \frac TM >0$. From now on, when no
confusion can arise, we use the shorthand notation 
$f^{m} := f(\cdot,t_m)$ for a function $f$ defined on $\overline I \times [0,T]$.

We propose the following filtered time-stepping scheme.
Set $x^0_h=I_h x_0$, where $I_h:[C^0(\overline I)]^n \to \Vh$ is the Lagrangian interpolation operator, and let $x^1_h \in \Vh$ be given. 
For
$m = 1,\ldots,M-1$, given $x^{m-1}_h, x^m_h \in \Vh$, let 
\begin{equation} \label{eq:defhatx}
\hat x^{m+1}_h:= 2 x^m_h - x^{m-1}_h.
\end{equation}
Then find $\bar x^{m+1}_h \in \Vh$ with
$\bar x^{m+1}_h - x^m_h \in \Vhpartial(\hat x^{m+1}_h)$, such that
\begin{subequations} \label{eq:feafilt}
\begin{align}
& \left(\frac{\bar x^{m+1}_h - x^m_h}{\Delta t} , \eta_h 
|\hat x^{m+1}_{h,\rho}|^2 \right)^h
+ \left(\bar x^{m+1}_{h,\rho} , \eta_{h,\rho} \right) = 0
\qquad \forall \eta_h \in \Vhpartial(\hat x^{m+1}_h) \label{eq:fea}
\end{align}
and update
\begin{align} \label{eq:update}
x^{m+1}_h:= \tfrac{2}{3} \bar x^{m+1}_h + \tfrac{2}{3} x^m_h - \tfrac{1}{3} x^{m-1}_h.
\end{align}
\end{subequations}

\begin{remark} \label{rem:fea}
The natural extension of the second order in time scheme from \cite{csftime} 
to the case of possibly open curves is given as follows.
First find
$x^{m+\frac12}_h\in\Vh$ with $x^{m+\frac12}_h - x^m_h \in \Vhpartial(x^m_h)$, 
such that
\begin{subequations} \label{eq:cnfea}
\begin{equation} \label{eq:cnpred}
\left(
\frac{x^{m+\frac12}_h-x^m_h}{\frac12\Delta t} , \eta_h |x^m_{h,\rho}|^2 \right)^h
+ \left( x^{m+\frac12}_{h,\rho} , \eta_{h,\rho} \right) = 0
\qquad \forall \eta_h \in \Vhpartial(x^m_h),
\end{equation}
and then find $x^{m+1}_h \in \Vh$ with 
$x^{m+1}_h - x^m_h \in \Vhpartial(x^{m+\frac12}_h)$ such that
\begin{equation} \label{eq:cncsfd}
\left( \frac{x^{m+1}_h-x^m_h}{\Delta t} , \eta_h 
|x^{m+\frac12}_{h,\rho}|^2 \right)^h
+ \tfrac12 \left(x^{m+1}_{h,\rho} + x^{m}_{h,\rho} , \eta_{h,\rho} \right) =  0
\qquad \forall \eta_h \in \Vhpartial(x^{m+\frac12}_h).
\end{equation}
\end{subequations}
Existence and uniqueness for \eqref{eq:cnfea}, as well as unconditional
stability, can be shown exactly as in \cite{csftime}. Unfortunately, it does
not seem to be possible to prove a quadratic convergence rate for the 
$L^2$--error for the scheme \eqref{eq:cnfea} in the case of open curves. 
In fact, some of our numerical results in
Section~\ref{sec:nr} indicate a suboptimal convergence rate for
\eqref{eq:cnfea} when the boundary $\partial\Omega$ is curved.
\end{remark}

For the subsequent analysis, it will be useful to 
introduce the abbreviations
\[
\hat f^{m+1}:= 2 f^m - f^{m-1}, \quad \tilde f^{m+1} := \tfrac{3}{2} f^{m+1} - f^m + \tfrac{1}{2} f^{m-1}, \quad
D_t f^{m+1} := \frac{3 f^{m+1}- 4f^m+ f^{m-1}}{2 \Delta t},
\]
so that \eqref{eq:feafilt} can be equivalently written in the form: Find $x^{m+1}_h \in \Vh$ such
that $D_t x^{m+1}_h \in \Vhpartial(\hat x^{m+1}_h)$ and
\begin{equation} \label{eq:fea1}
\left( D_t x^{m+1}_h, \eta_h 
|\hat x^{m+1}_{h,\rho}|^2 \right)^h
+ \left(\tilde x^{m+1}_{h,\rho} , \eta_{h,\rho} \right) = 0
\qquad \forall \eta_h \in \Vhpartial(\hat x^{m+1}_h).
\end{equation}

\begin{lemma} \label{lem:disctime}
Let $f: \overline I \times [0,T] \to \bR$. Then we have in $I$:
\begin{subequations}\label{eq:disctime}
\begin{align}
& D_t f^{m+1}  \cdot f^{m+1} = \frac{1}{4 \Delta t} \bigl( |  f^{m+1}  |^2 + |  \hat f^{m+2}  |^2 \bigr) \nonumber  \\ & \qquad \quad
- \frac{1}{4 \Delta t} \bigl( | f^{m}  |^2 + |  \hat  f^{m+1}  |^2 \bigr) + \frac{1}{4 \Delta t} | f^{m+1}  - 2 f^m  +  f^{m-1}  |^2,  \label{eq:disctime1}\\
& D_t f^{m+1} \cdot \hat f^{m+1} = \frac{1}{4 \Delta t} \bigl( 3 | f^{m+1}  |^2 - | f^m |^2 \bigr) \nonumber  \\ & \qquad\quad 
-  \frac{1}{4 \Delta t} \bigl( 3 | f^{m}  |^2 - | f^{m-1}  |^2 \bigr) - \frac{3}{4 \Delta t} | f^{m+1} - 2 f^m  +  f^{m-1} |^2, \label{eq:disctime2}\\
& D_t f^{m+1} \cdot \tilde f^{m+1} = \frac{1}{4 \Delta t} \bigl( | f^{m+1}  |^2 + |  \hat  f^{m+2}  |^2 + | f^{m+1}  - f^m |^2 \bigr) 
\nonumber  \\ & \qquad\quad 
- \frac{1}{4 \Delta t} \bigl( | f^{m}  |^2 + |  \hat  f^{m+1}  |^2 + | f^m  - f^{m-1}  |^2  \bigr) 
+ \frac{3}{4 \Delta t} | f^{m+1}  - 2 f^m  +  f^{m-1}  |^2. 
\label{eq:disctime3}
\end{align}
\end{subequations}
Analogous relations hold if the scalar product is replaced by a symmetric bilinear form.
\end{lemma}
\begin{proof}
The proof follows directly from elementary calculations. 
\end{proof}

We are now in a position to prove the well-posedness of \eqref{eq:feafilt}
together with an energy estimate, which can be seen as a discrete version of \eqref{eq:contstab}.

\begin{lemma} \label{lem:stab}
Suppose that $\hat x^{m+1}_{h,j} \in U$, $j=0,J$, and that
$| \hat x^{m+1}_{h,\rho} | >0$ in $I$. 
Then \eqref{eq:fea} has a unique solution $\bar x^{m+1}_h \in \Vh$ and the update $x^{m+1}_h$ from \eqref{eq:update} satisfies
\begin{align} 
& 4 \Delta t \left(|  D_t x^{m+1}_h |^2, | \hat x^{m+1}_{h,\rho} |^2 \right)^h + 
| x^{m+1}_h |_{1}^2 + | \hat x^{m+2}_h  |_{1}^2 + | x^{m+1}_h - x^m_h |_{1}^2  \nonumber \\ & \qquad 
\leq   | x^{m}_h |_{1}^2 + | \hat x^{m+1}_h  |_{1}^2 + | x^{m}_h - x^{m-1}_h |_{1}^2. \label{eq:stab}
\end{align}
\end{lemma}
\begin{proof} 
As \eqref{eq:fea} is a linear system with the same number of equations as
unknowns, existence follows from uniqueness. To show the latter, we
need to prove that the homogeneous system has only the trivial solution. Hence
let $\bar x_h \in \Vhpartial(\hat x^{m+1}_h)$ be such that
\begin{align*}
& \frac{1}{\Delta t}  \left( \bar x_h, \eta_h 
|\hat x^{m+1}_{h,\rho}|^2 \right)^h
+ \left(\bar x_{h,\rho} , \eta_{h,\rho} \right) = 0
\qquad \forall \eta_h \in \Vhpartial(\hat x^{m+1}_h). 
\end{align*}
Setting $\eta_h = \bar x_h$ we obtain
\begin{align*}
& \frac{1}{\Delta t}  \left( | \bar x_h |^2,
|\hat x^{m+1}_{h,\rho}|^2 \right)^h + | \bar x_h |_1^2  = 0,
\end{align*}
which implies that $\bar x_h=0$, as required. 
The estimate \eqref{eq:stab} follows by 
testing the equivalent formulation \eqref{eq:fea1} with 
$\eta_h = D_t x^{m+1}_h$ and using \eqref{eq:disctime3}.
\end{proof}

Our main result is the following optimal error estimate.
\begin{theorem} \label{thm:main}
Suppose that \eqref{eq:pde}
has a smooth solution on the time interval $[0,T]$ satisfying
\begin{equation} \label{eq:regul}
c_0 \leq | x_\rho | \leq C_0 \quad \mbox{ in } I \times [0,T]
\end{equation}
for some constants $c_0, C_0 \in \bRplus$. 
Let $x^0_h=I_h x_0$ and assume that $x^1_h \in \Vh$ is such that
\begin{equation} \label{eq:este1}
\| I_h x (\cdot,t_1) - x^1_h \|_{1}^2 \leq c(h^4 + (\Delta t)^4).
\end{equation}
Then there exist $h_0 > 0, \gamma>0$ such that if $0 < h \leq h_0$ and
$\Delta t \leq \gamma h^{\frac12}$,
then \eqref{eq:feafilt} has a unique 
solution $(x^{m}_h)_{m=2,\ldots,M}$, and the following error bounds hold:
\begin{subequations} \label{eq:h1h2}
\begin{align}
\max_{0 \leq m \leq M} \| x(\cdot,t_m) - x^m_h \|_0^2 
& \leq c \bigl( h^4 + (\Delta t)^4 \bigr);  \label{eq:h1} \\
 \max_{0 \leq m \leq M} | x(\cdot,t_m) - x^m_h |_1^2  & \leq c \bigl( h^2 + (\Delta t)^4 \bigr).  \label{eq:h2} 
\end{align}
\end{subequations}
\end{theorem}

\begin{remark} \label{rem:x1}
In practice, appropriate initial data $x^1_h$ satisfying \eqref{eq:este1} can,
for example, be obtained by using the backward Euler scheme from
\cite{DeckelnickE98} for $M_1 = \lceil \frac1{\Delta t} \rceil$ time steps 
with the artificial time step size $\Delta t / M_1$. Of course, for small
$\Delta t$ this soon becomes impractical.
An alternative is proposed in
Appendix~\ref{sec:AppA}, which requires the solution of only a single linear
system of equations.
\end{remark}

\setcounter{equation}{0}
\section{Proof of Theorem~\ref{thm:main}} \label{sec:ana}

From now on, and without loss of generality, we assume that $U$ and $F$ are
chosen such that
\begin{equation} \label{eq:gradF}
 |\nabla\,F(z)| =1 \quad \forall z \in \partial \Omega, \quad | \nabla F(z)| \geq \tfrac12 \quad \forall z \in U.
\end{equation}
We can therefore extend the definition of the projection operator $P$ to all of
$U$. In particular, 
we define the matrix valued functions $P, A: U \to \bR^{n \times n}$ by
\begin{equation} \label{eq:defPA}
P(z):= \Id - \frac{\nabla F(z)}{| \nabla F(z) |} \otimes \frac{\nabla F(z)}{| \nabla F(z) |}, \quad A(z):= \frac{1}{| \nabla F(z)|} P(z) D^2 F(z) P(z).
\end{equation}
By differentiating the condition $|\nabla F(z) | =1$ on $\partial \Omega$ 
in tangential direction we have
\begin{displaymath}
 P(z)  D^2 F(z) \nabla F(z) =0 \quad \forall z \in \partial \Omega.
\end{displaymath}
Let us fix $z \in \partial \Omega$ and $p \in B_\epsilon(z) \subset U$. Using the above identity together with the relation
$p-z= P(z)(p-z)+ ((p-z) \cdot \nabla F(z)) \nabla F(z)$, we obtain
\begin{align}
&
\frac{\nabla F(p)}{| \nabla F(p) |} - \nabla F(z) = \frac{\nabla F(p)}{| \nabla F(p) |} - \frac{\nabla F(z)}{| \nabla F(z) |} 
 = \frac{1}{| \nabla F(z)|} P(z) D^2 F(z)(p-z) + \mathcal O(|p-z|^2) \nonumber  \\ & \qquad
= \frac{1}{| \nabla F(z)|} P(z) D^2 F(z) P(z)(p-z) + \mathcal O(|p-z|^2) = A(z) (p-z) + \mathcal O(|p-z|^2). \label{eq:difnormal}
\end{align}

As a large part of the error analysis is concerned with handling the boundary conditions, it is convenient to interpret the scheme as a finite difference method.
To do so,  we view a function $v  \in \Vh$ as a grid function on $\mathcal G_h=\lbrace \rho_0, \rho_1,\ldots, \rho_J \rbrace$.
Setting $v_j=v(\rho_j)$ we introduce the finite difference operators
\begin{displaymath}
\delta^- v_j= \frac{v_j - v_{j-1}}{h}, \quad 
\delta^+ v_j= \frac{v_{j+1} - v_j}{h}, \quad 
\delta^2 v_j= \frac{\delta^+v_j - \delta^- v_j}{h}.
\end{displaymath}
For two grid functions $v,w: \mathcal G_h \to \bR^n$ one has the following summation by parts formula:
\begin{equation} \label{eq:sbp}
h \sum_{j=1}^J \delta^- v_j \cdot \delta^- w_j = - h \sum_{j=1}^{J-1} v_j \cdot \delta^2 w_j
+ \delta^- w_J \cdot v_J - \delta^+ w_0 \cdot v_0.
\end{equation}
In addition, we introduce the following discrete norms and seminorms
\begin{align} \label{eq:norms}
& |  v |_{0,h}^2 := \tfrac12 h |  v_0|^2 
+ h \sum_{j=1}^{J-1} |  v_j|^2 + \tfrac12 h |  v_J|^2; \quad
|  v |_{1,h}^2 := h \sum_{j=1}^J |\delta^-  v_j|^2; \nonumber \\
& \|  v \|_{1,h}^2 := |  v |_{0,h}^2 + |  v |_{1,h}^2; \quad
|  v |_{2,h}^2 := h \sum_{j=1}^{J-1} |\delta^2  v_j|^2.
\end{align}

We have the following lemma.
\begin{lemma} \label{lem:dsi}
Let $v: \mathcal G_h \to \bR^n$ be an arbitrary grid function. 
Then
\begin{subequations} \label{eq:lem21}
\begin{align} 
| v |_{0,h} \leq
\max_{0 \leq k \leq J} | v_k | & \leq |v_0| + | v |_{1,h}, 
\label{eq:dpi} \\
\max_{1\leq k \leq J} |\delta^- v_k| & \leq h^{-\frac12}
| v |_{1,h}, \label{eq:inv1} \\
\max_{0\leq k \leq J} |  v_k|^2 &
\leq |  v |_{0,h}^2 + 2 |  v |_{0,h} |  v |_{1,h}, \label{eq:dsi0} \\
\max_{1\leq k \leq J} |\delta^-  v_k|^2 &
\leq |  v |_{1,h}^2 + 2 |  v |_{1,h} |  v |_{2,h}. \label{eq:dsi}
\end{align}
\end{subequations}
\end{lemma}
\begin{proof} 
The first inequality in \eqref{eq:dpi} is trivial. In addition,
on noting that $v_k = v_0 + h\sum_{j=1}^k \delta^- v_j$
and using the Cauchy--Schwarz inequality, it holds that
\begin{equation*}
|v_k| \leq |v_0| + h \sum_{j=1}^k |\delta^- v_j| 
\leq |v_0| + h \sum_{j=1}^J |\delta^- v_j| \leq |v_0| + | v |_{1,h}, \quad 1 \leq k \leq J.
\end{equation*}
This proves the second inequality in \eqref{eq:dpi}. 
For the proofs of \eqref{eq:inv1}, \eqref{eq:dsi0} and \eqref{eq:dsi}, we refer
to Lemma~2.2 in \cite{degenD}. 
\end{proof}

In this section we will prove that if the assumptions of Theorem~\ref{thm:main}
are satisfied, then
\begin{equation} \label{eq:fdh1}
\max_{0 \leq m \leq M} \| x(\cdot,t_m) - x^m_h \|_{1,h}^2 
\leq c \bigl( h^4 + (\Delta t)^4 \bigr). 
\end{equation}
Clearly \eqref{eq:fdh1} implies a superconvergence result for the error
$\max_{0 \leq m \leq M} | I_h x(\cdot,t_m) - x^m_h |_{1}^2$, as well as the
error bounds \eqref{eq:h1h2}. \\
Let us define the grid functions $e^m_h: \mathcal G_h \to \bR^n$ by
\begin{equation} \label{eq:em} 
e^m_j = e^m_{h,j} =
 e^m_h(\rho_j)= x^m_{h,j} - x^m_j=x^m_{h,j}-x(\rho_j,t_m),  \quad j=0,\ldots,J,
\end{equation}
and set
\begin{equation} \label{eq:hate}
\hat e^{m+1}_h=2 e^m_h - e^{m-1}_h, \quad \tilde e^{m+1}_h=\tfrac32 e^{m+1}_h - e^m_h+ \tfrac12 e^{m-1}_h.
\end{equation}
Let $K>0$ and define for $m \in \lbrace 1,\ldots,M \rbrace$
\begin{align} \label{eq:defEm} 
E^m & := \tfrac{1}{4} \bigl(  | e^m_{h} |^2_{1,h} + | \hat e^{m+1}_h |_{1,h}^2 + | e^m_h - e^{m-1}_h |_{1,h}^2 \bigr) 
+ K \bigl(  | e^m_h |_{0,h}^2 + | \hat e^{m+1}_h |_{0,h}^2 \bigr)   + G^m_0 + G^m_J, 
\end{align}
where $G^m_0 = -(G^m_{0,1} + G^m_{0,2})$
and $G^m_J = G^m_{J,1} + G^m_{J,2}$, with
\begin{subequations} \label{eq:defGm}
\begin{align} \label{eq:defGm1}
G^m_{j,1} & := -    \bigl( x^{m}_\rho(\rho_j) \cdot \nabla F(x^{m}_j) \bigr) \Bigl( \tfrac{3}{4} \bigl(A(x^m_j)  e^m_j \cdot e^m_j \bigr)   - \tfrac{1}{4} \bigl( A(x^m_j)  e^{m-1}_{j} \cdot e^{m-1}_j\bigr)\Bigr) , \\
G^m_{j,2} & :=  a^{m}_j \cdot \bigl(  \tfrac32e^m_j - \tfrac12 e^{m-1}_j \bigr) = a^m_j \cdot \bigl( \tfrac12 e^m_j + \tfrac12 \hat e^{m+1}_j \bigr) , \label{eq:defGm2}
\end{align}
\end{subequations}
with $A$ as defined in \eqref{eq:defPA} and
\begin{align} 
a^{m}_j & := (\Delta t)^2  \bigl( x^{m}_\rho(\rho_j) \cdot \nabla F(x^{m}_j) \bigr) 
 A(x^m_j) x^{m-1}_{tt}(\rho_j)   
 + P(x^m_j) \bigl( \tfrac12(\Delta t)^2 x^{m}_{\rho t t}(\rho_j) + \tfrac16 h^2 x^{m}_{\rho \rho \rho}(\rho_j) \bigr) .\label{eq:defam}
 \end{align}
 The following lemma shows that $E^m$ controls the error in $\| \cdot \|_{1,h}$. 
\begin{lemma} Assume that $K$ satisfies $K \geq 18 K_0^2 +\tfrac32 K_0 + 3$ with $\displaystyle K_0:= \max_{\partial I \times [0,T]} |  x_\rho | \, | A(x)| $, where
$| A| $ denotes the spectral norm of $A$. Then
\begin{equation} \label{eq:Eequiv}
E^m \geq \tfrac{1}{16} \bigl( \| e^m_{h} \|_{1,h}^2 + \| \hat e^{m+1}_h \|_{1,h}^2 \bigr)  - c \bigl( h^4+ (\Delta t)^4 \bigr).
\end{equation}
\end{lemma}
\begin{proof} 
Observing that for a symmetric matrix $A$ and $v,w \in \bR^n$ we have
\begin{displaymath}
\tfrac34 Av \cdot  v - \tfrac14 A w \cdot w = - \tfrac14 Av \cdot v - \tfrac14 A(2v-w) \cdot (2v-w) + Av \cdot ( 2v -w),
\end{displaymath}
we find with the help of \eqref{eq:dsi0} and \eqref{eq:hate}
\begin{align*}
| G^m_{j,1} | & \leq K_0 \bigl( \tfrac14 | e^m_j |^2 + \tfrac14 |  \hat e^{m+1}_j  |^2 + | e^m_j | |  \hat e^{m+1}_j  | \bigr) \leq \tfrac34 K_0 \bigl(
 | e^m_j |^2 +  |  \hat e^{m+1}_j  |^2 \bigr) \\
& \leq \tfrac34 K_0 \bigl( | e^m_h |_{0,h}^2 + | \hat e^{m+1}_h  |_{0,h}^2 + 2 | e^m_h |_{0,h} | e^m_h |_{1,h} +2 |  \hat e^{m+1}_h  |_{0,h} |  \hat e^{m+1}_h  |_{1,h} \bigr) \\
& \leq \tfrac{1}{16} \bigl(  | e^m_h |_{1,h}^2 + |  \hat e^{m+1}_h  |_{1,h}^2 \bigr) + ( 9 K_0^2+ \tfrac34 K_0) \bigl(  | e^m_h |_{0,h}^2 + | \hat e^{m+1}_h  |_{0,h}^2 \bigr),
\end{align*}
while 
\[
| G^m_{j,2} | \leq c \bigl( h^2 + (\Delta t)^2 \bigr) \bigl( | e^m_j| + |   \hat e^{m+1}_j  | \bigr) \leq  \tfrac{1}{32} \bigl(  | e^m_h |_{1,h}^2 + | \hat e^{m+1}_h  |_{1,h}^2 \bigr)  
+ | e^m_h |_{0,h}^2 + | \hat e^{m+1}_h |_{0,h}^2 + c \bigl( h^4 + (\Delta t)^4 \bigr). 
\]
Hence
\[
| G^m_0| + | G^m_J | \leq \tfrac{3}{16}  \bigl(  | e^m_h |_{1,h}^2 + | \hat e^{m+1}_h  |_{1,h}^2 \bigr) + (18 K_0^2+ \tfrac32 K_0+2) \bigl(  | e^m_h |_{0,h}^2 + | \hat e^{m+1}_h  |_{0,h}^2 \bigr)+  c \bigl( h^4 + (\Delta t)^4 \bigr),
\]
which implies \eqref{eq:Eequiv}, on recalling our assumption on $K$.
\end{proof}

Our aim is to show by induction that $(x^{m}_h)_{m=2,\ldots,M}$
exists uniquely and that $x^m_h$ satisfies
\begin{equation} \label{eq:induction}
E^m \leq \hat c \bigl( h^4 + (\Delta t)^4 \bigr) e^{\mu t_m},
\end{equation}
provided that $0 <h \leq h_0$ and 
$0<\Delta t \leq \gamma h^{\frac12}$, for suitably chosen constants
$\hat c>0, \gamma>0, \mu>0$.
This would prove \eqref{eq:fdh1}, and hence \eqref{eq:h1h2}.
Since $e^0_h=0$, it follows from \eqref{eq:dsi0} and  \eqref{eq:este1} that
\[
E^1 \leq c \| e^1_h \|_{1,h} \Bigl( \| e^1_h \|_{1,h}
+ h^2 + (\Delta t)^2 \Bigr)
\leq c( h^4 + (\Delta t)^4) =: \hat c ( h^4 + (\Delta t)^4) \leq \hat c \bigl( h^4 + (\Delta t)^4 \bigr) e^{\mu t_1},
\]
so that the assertion holds for $m=1$. 
Suppose next that $x^m_h \in \Vh$ exists and \eqref{eq:induction} is satisfied for some $1 \leq m \leq M-1$.
We have from \eqref{eq:dsi0},  \eqref{eq:Eequiv} and \eqref{eq:induction}  that
\begin{align} \label{eq:embound} 
\max_{j=0,\ldots,J} | \hat e^{m+1}_j |   \leq c  \|  \hat e^{m+1}_h \|_{1,h}  \leq c  \bigl( \sqrt{E^m} + h^2 + (\Delta t)^2 \bigr)  \leq c e^{\frac12 \mu T} \bigl( h^2 + (\Delta t)^2 \bigr) \leq h,
\end{align}  
provided that $0 < h \leq h_0$ and $0<\Delta t \leq \gamma  h^{\frac{1}{2}}$ for a $\gamma>0$ sufficiently small. 
Similarly, on recalling \eqref{eq:inv1} we obtain
\begin{align*}  
&  \max_{j=1,\ldots,J} | \delta^- \hat e^{m+1}_j | \leq c h^{-\frac12} | \hat e^{m+1}_h |_{1,h} 
  \leq c h^{-\frac12} e^{\frac12 \mu T} \bigl( h^2 + (\Delta t)^2 \bigr) \leq c h^{-\frac12} e^{\frac12 \mu T} \bigl( h^2 + \gamma^2 h \bigr),
\end{align*} 
and hence with the help of the smoothness of $x$ and \eqref{eq:regul} that
\begin{equation} \label{eq:xhbound1}
\tfrac12 c_0 \leq | \delta^- \hat x^{m+1}_{h,j} | \leq 2 C_0, 
\quad j=1,\ldots,J,
\end{equation}
after choosing $h_0$ smaller if necessary. Since 
\begin{equation} \label{eq:hatxdif}
 \hat x^{m+1}_{h,j} - x^{m+1}_j = \hat e^{m+1}_j -(x^{m+1}_j - 2 x^m_j + x^{m-1}_j) , \quad j=0,1,\ldots,J,
\end{equation}
we deduce with the help of a compactness argument that $\hat x^{m+1}_{h,j} \in B_\epsilon(x^{m+1}_j) \subset U$ 
for $j=0,J$ uniformly in $m$ for $0<h \leq h_0$, provided that $h_0$ is sufficiently small. 
Hence
Lemma~\ref{lem:stab} implies the existence and uniqueness of the solution
$x^{m+1}_h$ to \eqref{eq:fea1}. Let us write \eqref{eq:fea1} as a finite difference scheme as follows:
\begin{subequations} \label{eq:DE98}
\begin{align}
D_t x^{m+1}_{h,j} - \dfrac{2}{  (\hat q^{m+1}_{h,j})^2 + (\hat q^{m+1}_{h,j+1})^2}\delta^2 \tilde x^{m+1}_{h,j} &=0, \quad j=1,\ldots,J-1,
\label{eq:DE98a} \\ 
D_t x^{m+1}_{h,j}  \cdot \nabla F(\hat x^{m+1}_{h,j}) & = 0, \quad j = 0, J,
\label{eq:DE98b} \\ 
P(\hat x^{m+1}_{h,0}) \bigl( (\hat q^{m+1}_{h,1})^{2} D_t x^{m+1}_{h,0} - \frac{2}{h} \delta^+  \tilde x^{m+1}_{h,0} \bigr)
 & = 0, 
 \label{eq:DE98c} \\ 
P(\hat x^{m+1}_{h,J})\bigl( (\hat q^{m+1}_{h,J})^{2}
 D_t x^{m+1}_{h,J} + \frac{2}{h} \delta^- \tilde x^{m+1}_{h,J} \bigr)
  & = 0, 
 \label{eq:DE98d} 
\end{align}
\end{subequations}
where we have abbreviated $\hat q^{m+1}_{h,j}:=| \delta^- \hat x^{m+1}_{h,j}|, j=1,\ldots,J$. 
Using \eqref{eq:DE98a} and \eqref{eq:csf} we derive the error relation 
\begin{align}
& D_t e^{m+1}_j  - \dfrac{2}{  (\hat q^{m+1}_{h,j})^2 + (\hat q^{m+1}_{h,j+1})^2}\,   \delta^2 \tilde e^{m+1}_{j}  \label{eq:erra2} \\  
& =   \left(  x_t(\rho_j,t_{m+1}) - D_t x^{m+1}_j \right)+ \frac{1}{ | x^{m+1}_\rho(\rho_j) |^2} \bigl(  \tilde x^{m+1}_{\rho \rho}(\rho_j)  - x^{m+1}_{\rho \rho}(\rho_j)  \bigr) \nonumber \\ & \qquad 
+ \frac{1}{ | x^{m+1}_\rho(\rho_j) |^2} \bigl( \delta^2 \tilde x^{m+1}_j - \tilde x^{m+1}_{\rho \rho}(\rho_j) \bigr)  +  \bigl(  \dfrac{2}{  (\hat q^{m+1}_{h,j})^2 + (\hat q^{m+1}_{h,j+1})^2} - \frac{1}{ | x^{m+1}_\rho(\rho_j) |^2} \bigr) \delta^2 \tilde x^{m+1}_j \nonumber \\
& =:\sum_{\ell=1}^4 g^{m+1}_{\ell,j}, \quad
 j=1,\ldots,J-1. \nonumber
\end{align}
Let us multiply by $- h \,   \delta^2 \tilde e^{m+1}_{j}$, sum over $j=1,\ldots,J-1$  and use \eqref{eq:sbp}:
\begin{align}
&   \sum_{j=1}^J h \,\delta^- D_t e^{m+1}_j  \cdot \delta^- \tilde e^{m+1}_j
+   \sum_{j=1}^{J-1} h \, \dfrac{2}{  (\hat q^{m+1}_{h,j})^2 + (\hat q^{m+1}_{h,j+1})^2} \,   |  \delta^2  \tilde e^{m+1}_{j}|^2 \label{eq:erra3}  \\ &  \qquad
=   D_t e^{m+1}_J  \cdot \delta^- \tilde e^{m+1}_J -  D_t e^{m+1}_0  \cdot \delta^+  \tilde e^{m+1}_0   
 - \sum_{\ell=1}^4 \sum_{j=1}^{J-1} h \,   g^{m+1}_{\ell,j} \cdot 
\delta^2 \tilde e^{m+1}_{j}. \nonumber
\end{align}
In view of \eqref{eq:disctime3}, and \eqref{eq:xhbound1}, we have
\begin{align} 
& \frac{1}{4 \Delta t} \bigl(   | e^{m+1}_h |_{1,h}^2 + | \hat e^{m+2}_h  |^2_{1,h} + | e^{m+1}_h - e^{m}_h |_{1,h}^2 \bigr) 
+ \frac{3}{4 \Delta t} | e^{m+1}_h - 2e^m_h + e^{m-1}_h |_{1,h}^2  + \frac{1}{4 C_0^2} | \tilde e^{m+1}_h |_{2,h}^2 \nonumber  \\
& \leq   \frac{1}{4 \Delta t} \bigl( | e^m_h |_{1,h}^2 + | \hat  e^{m+1}_h  |_{1,h}^2 + | e^m_h - e^{m-1}_h |_{1,h}^2 \bigr) 
 - \sum_{\ell=1}^4 \sum_{j=1}^{J-1} h \,  g^{m+1}_{\ell,j} \cdot 
\delta^2 \tilde e^{m+1}_{j} \nonumber \\
& \qquad +  D_t e^{m+1}_J  \cdot \delta^- \tilde e^{m+1}_J -  D_t e^{m+1}_0  \cdot \delta^+ \tilde e^{m+1}_0. \label{eq:erra4}
\end{align}
Using Taylor expansion, it is not difficult to see that 
\begin{equation} \label{eq:g123}
\sum_{\ell=1}^3 | g^{m+1}_{\ell,j} | \leq c  \bigl( h^2 + (\Delta t)^2 \bigr), \quad j=1,\ldots,J-1,
\end{equation}
while 
\begin{align}
| g^{m+1}_{4,j} | &\leq c \bigl| \dfrac{2}{  (\hat q^{m+1}_{h,j})^2 + (\hat q^{m+1}_{h,j+1})^2} - \dfrac{2}{  ( q^{m+1}_{j})^2 + (q^{m+1}_{j+1})^2} \bigr| + c \bigl| \dfrac{2}{  ( q^{m+1}_{j})^2 + (q^{m+1}_{j+1})^2}  - \frac{1}{| x^{m+1}_\rho(\rho_j)|^2} \bigr| \nonumber \\
& \leq c \bigl( | \delta^-(\hat x^{m+1}_{h,j} - x^{m+1}_j) | + | \delta^-(\hat x^{m+1}_{h,j+1}- x^{m+1}_{j+1}) | \bigr) + c | (q^{m+1}_j)^2 + (q^{m+1}_{j+1})^2 -   2 | x^{m+1}_\rho(\rho_j)|^2  | \nonumber \\
& \leq c  \bigl( | \delta^- \hat e^{m+1}_j | + | \delta^- \hat e^{m+1}_{j+1} | \bigr) + c \bigl( h^2 + (\Delta t)^2 \bigr), \quad j=1,\ldots,J-1,  \label{eq:g4} 
\end{align}
where we used \eqref{eq:hatxdif} and \cite[(3.19a)]{csftime} in the last step. As a result, we obtain
\begin{align}  
& | \sum_{\ell=1}^4 \sum_{j=1}^{J-1} h \,   g^{m+1}_{\ell,j} \cdot \delta^2 \tilde e^{m+1}_{j} | \leq c | \tilde e^{m+1}_h |_{2,h} \Bigl( \sum_{\ell=1}^4  \sum_{j=1}^{J-1} h 
| g^{m+1}_{\ell,j} |^2 \Bigr)^{\frac{1}{2}}  \nonumber  \\ & \quad
\leq c | \tilde e^{m+1}_h  |_{2,h} \bigl( | \hat e^{m+1}_h |_{1,h}  + h^2 + (\Delta t)^2 \bigr)  
\leq \frac{1}{8 C_0^2} | \tilde e^{m+1}_h |_{2,h}^2 + c \bigl(  | \hat e_h^{m+1} |_{1,h}^2 +  h^4 + (\Delta t)^4 \bigr). \label{eq:rhs1}
\end{align}
Let us next examine the boundary terms and split 
\begin{align}\label{eq:t1t2} 
 D_t  e^{m+1}_J  \cdot \delta^- \tilde e^{m+1}_J
& = \bigl(  D_t e^{m+1}_J  \cdot \hat d^{m+1}_J  \bigr) \bigl( \delta^- \tilde e^{m+1}_J \cdot \hat d^{m+1}_J  \bigr) 
 + \hat P^{m+1}_J D_t e^{m+1}_J  \cdot \hat P^{m+1}_J  \delta^- \tilde e^{m+1}_J \nonumber\\ &
 =:T_1+ T_2, 
\end{align}
where we have abbreviated $\hat d^{m+1}_J:= \frac{\nabla F(\hat x^{m+1}_{h,J})}{| \nabla F(\hat x^{m+1}_{h,J})|}$ and $\hat P^{m+1}_J:=P(\hat x^{m+1}_{h,J})$.
Using \eqref{eq:DE98b} and \eqref{eq:pdeb}, we have
\[
D_t e^{m+1}_J  \cdot \hat d^{m+1}_J 
 = D_t x^{m+1}_J  \cdot \bigl( \frac{\nabla F(x^{m+1}_J)}{| \nabla F(x^{m+1}_J) |}  -\frac{\nabla F(\hat x^{m+1}_{h,J})}{| \nabla F(\hat x^{m+1}_{h,J})|}   \bigr)
- \bigl( D_t x^{m+1}_J - x^{m+1}_t(\rho_J) \bigr)  \cdot  \nabla F(x^{m+1}_J),
\]
so that \eqref{eq:hatxdif} together with a Taylor expansion yields
\begin{align} \label{eq:edif}
&\left|  D_t e^{m+1}_J  \cdot \hat d^{m+1}_J \right|  \leq c  \bigl(  | \hat e^{m+1}_J |  + (\Delta t)^2 \bigr). 
\end{align} 
Recalling \eqref{eq:dsi} and the relation $\tilde e^{m+1}_h= \tfrac32 e^{m+1}_h - \tfrac12 \hat e^{m+1}_h$, we deduce that 
\[
\left|  \delta^- \tilde e^{m+1}_J  \cdot \hat d^{m+1}_J  \right| \leq \left|  \delta^- \tilde e^{m+1}_J  \right|  \leq c\bigl(   | e^{m+1}_h  |_{1,h} + | \hat e^{m+1}_h |_{1,h} \bigr)
+ c | \tilde e^{m+1}_h|_{2,h}  ,
\]
and therefore
\begin{align}
| T_1 | & \leq c \bigl( | \hat e^{m+1}_J|   + (\Delta t)^2 \bigr) \bigl( | e^{m+1}_h |_{1,h} + | \hat e^{m+1}_h |_{1,h} + | \tilde e^{m+1}_h |_{2,h}  \bigr)  \nonumber \\
& \leq \epsilon  | \tilde e^{m+1}_h |_{2,h}^2 + c_\epsilon  \bigl(  | e^{m+1}_h |_{1,h}^2 +  | \hat e^{m+1}_h |_{1,h}^2  + | \hat e^{m+1}_J|^2 + (\Delta t)^4 \bigr).
\label{eq:t1}
\end{align}
In order to deal with $T_2$ we first note that in view of \eqref{eq:DE98d} 
\begin{equation} \label{eq:t2a}
 \hat P^{m+1}_J \delta^- \tilde e^{m+1}_J  
 =  - \tfrac12 h  (\hat q^{m+1}_{h,J})^2  \hat P^{m+1}_J D_t x^{m+1}_{h,J} 
  - \hat P^{m+1}_J \delta^- \tilde x^{m+1}_J .  
\end{equation}
Taylor expansion together with \eqref{eq:csf} yields
\begin{align*}
&  \delta^- \tilde x^{m+1}_J = \tilde x^{m+1}_\rho(\rho_J)  -\tfrac12 h \tilde x^{m+1}_{\rho \rho}(\rho_J)  
+ \tfrac16 h^2 \tilde x^{m+1}_{\rho \rho \rho}(\rho_J)  + \mathcal O(h^3) \\
& =  x^{m+1}_\rho(\rho_J)+ \tfrac12 (\Delta t)^2 x^{m+1}_{\rho t t}(\rho_J)  - \tfrac12 h x^{m+1}_{\rho \rho}(\rho_J) + \tfrac16 h^2 x^{m+1}_{\rho \rho \rho}(\rho_J) + \mathcal O(h^3 +  (\Delta t)^3) \\
& =   x^{m+1}_\rho(\rho_J)  - \tfrac12 h | x^{m+1}_\rho (\rho_J) |^2 x^{m+1}_t(\rho_J) + \tfrac12 (\Delta t)^2 x^{m+1}_{\rho t t}(\rho_J) + \tfrac16 h^2 x^{m+1}_{\rho \rho \rho}(\rho_J) + \mathcal O(h^3 + (\Delta t)^3) \\
& =   x^{m+1}_\rho(\rho_J)  - \tfrac12 h (q^{m+1}_J)^2 D_t x^{m+1}_J + r^{m+1}_J + \mathcal O(h^3 + (\Delta t)^3),
\end{align*}
where we have abbreviated
\begin{equation} \label{eq:defd}
r^{m+1}_J:= \tfrac12 (\Delta t)^2 x^{m+1}_{\rho t t}(\rho_J) + \tfrac16 h^2 x^{m+1}_{\rho \rho \rho}(\rho_J)
\end{equation}
and used the estimate
\begin{displaymath}
| (q^{m+1}_J)^2 - | x^{m+1}_\rho(\rho_J) |^2 | \leq c h^2,
\end{displaymath}
which follows from the fact that $x^{m+1}_\rho(\rho_J)  \cdot x^{m+1}_{\rho \rho}(\rho_J)= | x^{m+1}_\rho(\rho_J) |^2 x^{m+1}_\rho(\rho_J) \cdot x^{m+1}_t(\rho_J)=0$.  Setting $P^{m+1}_J:= P(x^{m+1}_J)$ we infer  
with the help of \eqref{eq:pdec} that 
\begin{align*}
\hat P^{m+1}_J \delta^- \tilde x^{m+1}_J& =\hat  P^{m+1}_J x^{m+1}_\rho(\rho_J)  - \tfrac12  h (q^{m+1}_J)^2 \hat P^{m+1}_J D_t x^{m+1}_J   
    + \hat P^{m+1}_J r^{m+1}_J 
+ \mathcal O(h^3 + (\Delta t)^3) \\
& =  \bigl( \hat P^{m+1}_J - P^{m+1}_J \bigr)  x^{m+1}_\rho(\rho_J)- \tfrac12  h (q^{m+1}_J)^2 \hat P^{m+1}_J D_t x^{m+1}_J  
+       P^{m+1}_J r^{m+1}_J    + R^{(1)}_J,
\end{align*}
where $| R^{(1)}_J| \leq c( (h^2+(\Delta t)^2) | \hat e^{m+1}_J| + h^3 +(\Delta t)^3)$. 
In order to handle the first term on the right hand side, we use \eqref{eq:difnormal}, \eqref{eq:hatxdif} and again \eqref{eq:pdec} to obtain
\begin{align} \label{eq:difproj}
& \bigl(\hat P^{m+1}_J - P^{m+1}_J \bigr)  x^{m+1}_\rho(\rho_J) \nonumber \\
& =  - \bigl(   \frac{\nabla F(\hat x^{m+1}_{h,J})}{| \nabla F(\hat x^{m+1}_{h,J})|} \cdot x^{m+1}_\rho(\rho_J) \bigr)
  \frac{\nabla F(\hat x^{m+1}_{h,J})}{| \nabla F(\hat x^{m+1}_{h,J})|} +
\bigl( \nabla F(x^{m+1}_J)  \cdot x^{m+1}_\rho(\rho_J) \bigr)   \nabla F(x^{m+1}_J) \nonumber  \\
& = -  \Bigl( \bigl(   \frac{\nabla F(\hat x^{m+1}_{h,J})}{| \nabla F(\hat x^{m+1}_{h,J})|} - \nabla F(x^{m+1}_J) \bigr) \cdot x^{m+1}_\rho(\rho_J)
\Bigr)  \frac{\nabla F(\hat x^{m+1}_{h,J})}{| \nabla F(\hat x^{m+1}_{h,J})|} \nonumber  \\
& \quad - \bigl( \nabla F(x^{m+1}_J)  \cdot x^{m+1}_\rho(\rho_J) \bigr) 
\bigl( \frac{\nabla F(\hat x^{m+1}_{h,J})}{| \nabla F(\hat x^{m+1}_{h,J})|} - \nabla F(x^{m+1}_J)   \bigr) \nonumber  \\
& = - \bigl( \nabla F(x^{m+1}_J) \cdot x^{m+1}_\rho(\rho_J) \bigr)
A(x^{m+1}_J)  (\hat x^{m+1}_{h,J}- x^{m+1}_J) + \mathcal O(| \hat x^{m+1}_{h,J} - x^{m+1}_J|^2) \nonumber\\
& =  -  \bigl(  \nabla F(x^{m+1}_J) \cdot x^{m+1}_\rho(\rho_J) \bigr) A(x^{m+1}_J) \bigl( \hat  e^{m+1}_J -  (\Delta t)^2  x^m_{tt}(\rho_J) \bigr)  +R^{(2)}_J  ,
\end{align} 
where $| R^{(2)}_J | \leq c    \bigl(   | \hat e^{m+1}_J |^2  + (\Delta t)^4 \bigr)$.  Thus 
\begin{align*}
\hat P^{m+1}_J \delta^- \tilde x^{m+1}_J & = - \tfrac12  h (q^{m+1}_J)^2 \hat P^{m+1}_J D_t x^{m+1}_J  -  \bigl(  \nabla F(x^{m+1}_J) \cdot x^{m+1}_\rho(\rho_J) \bigr) A(x^{m+1}_J)  \hat  e^{m+1}_J \\
& \quad + (\Delta t)^2 \bigl( x^{m+1}_\rho(\rho_J) \cdot  \nabla F(x^{m+1}_J)  \bigr) A(x^{m+1}_J)  x^m_{tt}(\rho_J) +       P^{m+1}_J r^{m+1}_J    + R^{(3)}_J,
\end{align*}
where $|  R^{(3)}_J | \leq c \bigl( | \hat e^{m+1}_J |^2   + h^3 + (\Delta t)^3 \bigr)$. 
Inserting the above relations into \eqref{eq:t2a}, we obtain
with the help of \eqref{eq:defam} that
\begin{align} 
\hat  P^{m+1}_J \delta^- \tilde e^{m+1}_J   
&= -\tfrac12 h (\hat q^{m+1}_{h,J})^2  \hat P^{m+1}_J D_t  e^{m+1}_{J} 
   + \tfrac12 h \bigl( (q^{m+1}_J)^2 - (\hat q^{m+1}_{h,J})^2 \bigr) \hat P^{m+1}_J D_t x^{m+1}_J   \nonumber  \\
   &  \quad -a^{m+1}_J 
  +    \bigl( x^{m+1}_\rho(\rho_J) \cdot \nabla F(x^{m+1}_J) \bigr) A(x^{m+1}_J) \hat e^{m+1}_J     -R^{(3)}_J . \label{eq:t2c} 
\end{align}
As a result, recalling \eqref{eq:disctime2} and observing that
$\Delta t D_t e^{m+1}_J = (\frac{3}{2} e^{m+1}_J - \frac{1}{2} e^m_J)- (\frac{3}{2} e^m_J - \frac{1}{2} e^{m-1}_J)$, 
we obtain
\begin{align*}
T_2 &=  \hat P^{m+1}_J D_t e^{m+1}_J  \cdot \hat P^{m+1}_J  \delta^- \tilde e^{m+1}_J 
 \leq -\tfrac12  h (\hat q^{m+1}_{h,J})^2 \left|  \hat P^{m+1}_J D_t e^{m+1}_{J}   \right|^2 - a^{m+1}_J  \cdot P^{m+1}_J  D_t e^{m+1}_J  \\
& \quad +  \bigl( x^{m+1}_\rho(\rho_J) \cdot \nabla F(x^{m+1}_J) \bigr) \bigl( A(x^{m+1}_J) \hat e^{m+1}_J \bigr) \cdot 
 \bigl( P^{m+1}_J D_t e^{m+1}_J \bigr) \\
& \quad + c \bigl( | \hat e^{m+1}_J |^2 + h^3  +  (\Delta t)^3+  h  | \delta^- \hat e^{m+1}_J | \bigr)  | D_t e^{m+1}_J |    \\
& = -\tfrac12  h (\hat q^{m+1}_{h,J})^2 \left|  \hat P^{m+1}_J D_t e^{m+1}_{J}   \right|^2 - a^{m+1}_J  \cdot  D_t e^{m+1}_J  \\
&  \quad +  \bigl( x^{m+1}_\rho(\rho_J) \cdot \nabla F(x^{m+1}_J) \bigr) \bigl( A(x^{m+1}_J) \hat e^{m+1}_J \bigr) \cdot 
  D_t e^{m+1}_J  \\
 & \quad + c \bigl( | \hat e^{m+1}_J |^2 + h^3  +  (\Delta t)^3+  h  | \delta^- \hat e^{m+1}_J | \bigr)  | D_t e^{m+1}_J |    \\ 
&  \leq  - \tfrac12 h  (\hat q^{m+1}_{h,J})^2 \left| \hat P^{m+1}_J D_t  e^{m+1}_{J}    \right|^2 - \frac{1}{\Delta t} \bigl( G^{m+1}_{J,2} - G^{m}_{J,2} + G^{m+1}_{J,1} - G^{m}_{J,1} \bigr)\\
&\quad + c \bigl( h^4 + (\Delta t)^4 + | e^{m-1}_J|^2 + | e^m_J |^2 \bigr) + \frac{c}{\Delta t}  | e^{m+1}_J - 2 e^m_J + e^{m-1}_J |^2 \\
& \quad + c \bigl( h^3  +  (\Delta t)^3+ h  | \hat e^{m+1}_J |+   h  | \delta^- \hat e^{m+1}_J | \bigr)  | D_t e^{m+1}_J | ,
\end{align*}
where we also used  \eqref{eq:embound}. Recalling \eqref{eq:xhbound1} and \eqref{eq:edif} and the fact that $\Delta t \leq \gamma h^{\frac{1}{2}}$, we infer that 
\begin{align}
T_2 & \leq  - \tfrac18  h c_0^2  \left|  D_t e^{m+1}_{J} \right|^2 + c h \left|  D_t e^{m+1}_{J} \cdot \hat d^{m+1}_J \right|^2 - \frac{1}{\Delta t} \bigl( G^{m+1}_J - G^{m}_J \bigr)
+ \frac{c}{\Delta t} | e^{m+1}_J - 2 e^m_J + e^{m-1}_J |^2 \nonumber \\
& \quad + c  \bigl( h^4 +  (\Delta t)^4 + h^{-1}(\Delta t)^6 + | e^{m-1}_J|^2 + | e^m_J |^2 \bigr) +\tfrac{1}{16} h c_0^2  | D_t e^{m+1}_J |^2+ 
c h\bigl( | \hat e^{m+1}_J|^2 + | \delta^- \hat e^{m+1}_J |^2 \bigr) \nonumber \\
& \leq - \tfrac{1}{16} h c_0^2 | D_t e^{m+1}_J |^2 - \frac{1}{\Delta t} \bigl( G^{m+1}_J - G^{m}_J  \bigr)
+ \frac{c}{\Delta t}  | e^{m+1}_J - 2 e^m_J + e^{m-1}_J |^2 \nonumber  \\
& \quad +  c h| \delta^- \hat e^{m+1}_J |^2  + c  \bigl( h^4 + (\Delta t)^4 + | \hat e^{m+1}_J|^2 + | e^m_J |^2 \bigr). \label{eq:t2}
\end{align}
Arguing in the same way for the left end point and inserting \eqref{eq:t1}, \eqref{eq:t2} into \eqref{eq:t1t2}, and thus into \eqref{eq:erra4}, and applying \eqref{eq:dsi0}, we obtain
\begin{align} 
& \tfrac1{16} c_0^2 h \Delta t \bigl(  \left|  D_t e^{m+1}_{0}  \right|^2 +  \left| D_t  e^{m+1}_{J} \right|^2 \bigr) + G^{m+1}_0 + G^{m+1}_J  + \frac{\Delta t}{16 C_0^2} | \tilde e^{m+1}_h  |_{2,h}^2 \nonumber \\
& \quad  +  \tfrac{1}{4} \bigl(   | e^{m+1}_h |_{1,h}^2 + | \hat e^{m+2}_h  |_{1,h}^2 +| e^{m+1}_h - e^m_h |_{1,h}^2 \bigr)  + \tfrac{1}{4} | e^{m+1}_h - 2e^m_h + e^{m-1}_h |_{1,h}^2  \nonumber \\
& \leq G^{m}_0 +  G^{m}_J + \tfrac{1}{4} \bigl(   |  e^{m}_h |_{1,h}^2 + |  \hat e^{m+1}_h  |_{1,h}^2 + | e^m_h - e^{m-1}_h |_{1,h}^2 \bigr)    + c \max_{j=0,J} | e^{m+1}_j - 2 e^m_j + e^{m-1}_j |^2 \nonumber \\
& \quad    + c \Delta t \bigl( \| e^{m+1}_h \|_{1,h}^2+ \| e^m_h \|_{1,h}^2 + \| \hat e^{m+1}_h \|_{1,h}^2 \bigr)  + c \Delta t \bigl( h^4 + (\Delta t)^4 \bigr).  \label{eq:erra5} 
\end{align}
Next, we infer from \eqref{eq:erra2}, \eqref{eq:xhbound1}, 
\eqref{eq:g123} and \eqref{eq:g4} that
\begin{align} \label{eq:gest} 
\Delta t \sum_{j=1}^{J-1} h | D_t e^{m+1}_j  |^2 & \leq \Delta t   \sum_{j=1}^{J-1}  h  \left( 2  \bigl(  \dfrac{2}{  (\hat q^{m+1}_{h,j})^2 + (\hat q^{m+1}_{h,j+1})^2} \bigr)^2 | \delta^2 \tilde e^{m+1}_j |^2 +
2 \Bigl| \sum_{\ell=1}^4 g^{m+1}_{\ell,j} \Bigr|^2 \right) \nonumber \\
& \leq \frac{32}{c_0^4} \Delta t | \tilde e^{m+1}_h  |_{2,h}^2 + c \Delta t  \sum_{\ell=1}^4 \sum_{j=1}^{J-1} h | g^{m+1}_{\ell,j} |^2 \nonumber \\
& \leq \frac{32}{c_0^4} \Delta t | \tilde e^{m+1}_h |_{2,h}^2   + c \Delta t | \hat e^{m+1}_h |_{1,h}^2  + c \Delta t \bigl( h^4 + (\Delta t)^4 \bigr).
\end{align}
In view of \eqref{eq:erra5} and \eqref{eq:gest}, there exists $c_1>0$, which depends on $c_0$ and $C_0$, such that
\begin{align} \label{eq:erra6} 
& c_1 \Delta t  | D_t e^{m+1}_h  |_{0,h}^2  +  \tfrac{1}{4} \bigl(   | e^{m+1}_h |_{1,h}^2 + | \hat e^{m+2}_h  |_{1,h}^2 +| e^{m+1}_h - e^m_h |_{1,h}^2 \bigr) \nonumber \\
& \quad  + G^{m+1}_0 + G^{m+1}_J + \tfrac{1}{4} | e^{m+1}_h - 2e^m_h + e^{m-1}_h |_{1,h}^2 \nonumber \\
& \leq \tfrac{1}{4} \bigl(  |  e^{m}_h |_{1,h}^2 + | \hat e^{m+1}_h  |_{1,h}^2 + | e^m_h - e^{m-1}_h |_{1,h}^2 \bigr)  +  G^{m}_0 +  G^{m}_J  + c \max_{j=0,J} | e^{m+1}_j - 2 e^m_j + e^{m-1}_j |^2 \nonumber \\
& \quad    + c \Delta t \bigl(\| e^{m+1}_h \|_{1,h}^2+  \| e^m_h \|_{1,h}^2 + \| \hat e^{m+1}_h \|_{1,h}^2 \bigr)  + c \Delta t \bigl( h^4 + (\Delta t)^4 \bigr) \nonumber \\
& \leq \tfrac{1}{4} \bigl(  |  e^{m}_h |_{1,h}^2 + | \hat e^{m+1}_h  |_{1,h}^2 + | e^m_h - e^{m-1}_h |_{1,h}^2 \bigr)  +  G^{m}_0 +  G^{m}_J + \tfrac{1}{8}  | e^{m+1}_h - 2e^m_h + e^{m-1}_h |_{1,h}^2 \nonumber    \\
& \quad + c  | e^{m+1}_h - 2e^m_h + e^{m-1}_h |_{0,h}^2 + c \Delta t  \bigl(  \|  e^{m+1}_h \|_{1,h}^2 + \| e^m_h \|_{1,h}^2 + \| \hat e^{m+1}_h \|_{1,h}^2 \bigr)  + c \Delta t \bigl( h^4 + (\Delta t)^4 \bigr),
\end{align} 
where we have used \eqref{eq:dsi0} and Young's inequality. 
Furthermore, we have in view of \eqref{eq:disctime1} that
\begin{align*}
& K \bigl( | e^{m+1}_h |_{0,h}^2 + | \hat e^{m+2}_h  |^2_{0,h} + | e^{m+1}_h - 2 e^m_h + e^{m-1}_h |_{0,h}^2 \bigr) \\
& = K \bigl(  | e^{m}_h |_{0,h}^2  + | \hat e^{m+1}_h  |_{0,h}^2 \bigr) + 4 K  \Delta t \bigl(e^{m+1}_h, D_t e^{m+1}_h\bigr)^h \\
& \leq  K \bigl(  | e^{m}_h |_{0,h}^2  + | \hat e^{m+1}_h  |_{0,h}^2 \bigr) +   \tfrac12 c_1 \Delta t  | D_t e^{m+1}_h  |_{0,h}^2 + c \Delta t  | e^{m+1}_h |_{0,h}^2,
\end{align*}
which combined with \eqref{eq:erra6}, and after choosing $K$ larger if 
necessary, yields
\begin{align*}
E^{m+1}&  \leq E^{m} + c \Delta t \bigl( \| e^{m+1}_h \|_{1,h}^2 + \| e^m_h \|_{1,h}^2 + \| \hat e^{m+1}_h \|_{1,h}^2 \bigr)  + c \Delta t \bigl( h^4 + (\Delta t)^4 \bigr) \\
& \leq E^{m} + c \Delta t \bigl( E^{m+1} + E^m  \bigr) + c \Delta t \bigl( h^4 + (\Delta t)^4 \bigr).
\end{align*}
On recalling the induction hypothesis \eqref{eq:induction}, we obtain for small $\Delta t$ that
\begin{align*}
E^{m+1}& \leq E^{m} + c \Delta t E^m + c \Delta t \bigl( h^4 + (\Delta t)^4 \bigr) \\
& \leq \hat c \bigl( h^4 + (\Delta t)^4 \bigr) e^{\mu t_{m}} + c  \hat c \Delta t \bigl( h^4 + (\Delta t)^4 \bigr)  e^{\mu t_m}  + c \Delta t \bigl( h^4 + (\Delta t)^4 \bigr) \\
& \leq \hat c \bigl( h^4 + (\Delta t)^4 \bigr) e^{\mu t_{m}}  \bigl( 1+  2 c \Delta t e^{\mu \Delta t}  \bigr)  \leq \hat c \bigl( h^4 + (\Delta t)^4 \bigr)  e^{\mu t_{m+1}},
\end{align*}
provided that we choose $\hat c \geq c$ and $\mu \geq 2c$. 
Thus \eqref{eq:induction} holds for $m+1$ and hence for all $0 \leq m \leq M$. 
In conclusion, it follows from \eqref{eq:induction} and \eqref{eq:Eequiv} 
that \eqref{eq:fdh1} holds. This completes the proof.

\setcounter{equation}{0}
\section{Numerical results} \label{sec:nr}
\newcommand{\errorxL}{\| x -  x_h\|_{0}}
\newcommand{\errorxH}{\| x -  x_h\|_{1}}

Unless otherwise stated, we use the scheme \eqref{eq:feafilt}
for the numerical simulations presented in this section.  
In all these experiments we make use of the initial data $x^1_h$ as the solution of \eqref{eq:DE98prednew}.
For some of our simulations we will monitor the ratio
\begin{equation} \label{eq:ratio}
\ratio^m = \dfrac{\max_{j=1,\ldots,J} |x^m_{h,j} - x^m_{h,j-1}|}
{\min_{j=1,\ldots, J} |x^m_{h,j} - x^m_{h,j-1}|}
\end{equation}
between the lengths of the longest and shortest element of the polygonal curve
$x^m_h(\overline I)$.
Clearly $\ratio^m\geq1$, with equality if and only if the curve is 
equidistributed.

\subsection{Open curves}

In \cite{DeckelnickE98} an arclength solution to \eqref{eq:pde} consisting of
shrinking half-circles within a half plane is considered. 
For the right half plane $\Omega = \bRplus \times \bR$ we generalize it here
to
\begin{equation} \label{eq:csfsolxg}
x(\rho,t) = (1 - 2 t)^{\frac12} \left( \sin g_1(\rho), \cos g_1(\rho) \right)^T
\quad \rho \in I,\ t \in [0, \tfrac12),
\end{equation}
with
\begin{equation} \label{eq:g}
g_1(\rho) = \pi \rho + \delta \sin(\pi\rho), \quad \delta = 0.1.
\end{equation}
Clearly, this proposed solution no longer solves \eqref{eq:pde}, but rather an
inhomogeneous variant with a nonzero right hand side in \eqref{eq:csf}. Hence,
for the convergence experiments, we compare \eqref{eq:csfsolxg} with the
discrete solutions of \eqref{eq:feafilt}, where the zero 
right hand side in \eqref{eq:fea} is replaced with $(f^{m+1}, \eta_h)$, 
where 
\begin{equation} \label{eq:f}
f = |x_\rho|^2 x_t - x_{\rho\rho}
\end{equation}
denotes the residual of \eqref{eq:csfsolxg} with respect to \eqref{eq:csf}. 
In a similar fashion, we replace the right hand sides in 
\eqref{eq:DE98prednewa}, \eqref{eq:DE98prednewd}, \eqref{eq:DE98prednewe} 
with $f^{0}_j$, $P(x^0_0)f^{0}_0$ and
$P(x^0_J) f^{0}_J$, respectively.
The results for the scheme \eqref{eq:feafilt}, 
are shown in Table~\ref{tab:fd_filteredsemicircleg},
where we have used the error notations
\[
\errorxL := \max_{0 \leq m \leq M} \| x(\cdot,t_m) - x^m_h \|_{0} , \quad
\errorxH := \max_{0 \leq m \leq M} \| x(\cdot,t_m) - x^m_h \|_{1} .
\]
These results confirm the optimal error estimates proven in Theorem~\ref{thm:main}.
\begin{table}
\center
\begin{tabular}{|r|c|c|c|c|}
\hline
$J$ & $\errorxL$ & EOC & $\errorxH$ & EOC \\ \hline
32   &9.9804e-03& ---&9.0571e-02& --- \\
64   &2.7535e-03&1.86&4.5289e-02&1.00 \\
128  &7.2150e-04&1.93&2.2645e-02&1.00 \\
256  &1.7528e-04&2.04&1.1323e-02&1.00 \\
512  &4.3186e-05&2.02&5.6613e-03&1.00 \\
1024 &1.0851e-05&1.99&2.8307e-03&1.00 \\
2048 &2.7194e-06&2.00&1.4153e-03&1.00 \\
4096 &6.7858e-07&2.00&7.0766e-04&1.00 \\
\hline
\end{tabular}
\caption{Errors for the convergence test for \eqref{eq:csfsolxg} with
\eqref{eq:g} over the time interval $[0,0.4]$. 
We use $\Delta t = h$ for the scheme \eqref{eq:feafilt}.
We also display the experimental orders of convergence (EOC).
}
\label{tab:fd_filteredsemicircleg}
\end{table}%

Next we would like to construct a forced solution within the elliptic domain
\begin{equation} \label{eq:Omega_ell}
\Omega = \{ \tbinom{x}{y} \in \bR^2 : \tfrac14 x^2 + y^2 < 1\}.
\end{equation}
To this end, we define a family of circle segments that meet the boundary
$\partial\Omega$ orthogonally. Let
$\alpha : [0,T] \to (0,1)$ and set 
$\beta(t) = \sqrt{4 - 3 \alpha^2(t)}$.
We postulate
\begin{subequations} \label{eq:xnew_ellipse}
\begin{equation} \label{eq:xnew_ellipsea}
x(\rho, t) = (1 - \alpha^2(t))^{-\frac12} \binom{\alpha(t) \beta(t) \cos(g_2(\rho,t))}
{1 + \alpha(t) \beta(t) \sin(g_2(\rho,t))},
\quad \rho \in I,\ t \in [0, T],
\end{equation}
where
\begin{equation} \label{eq:gnew_ellipse}
g_2(\rho,t) = (2\rho - 1)\arccos\left(\frac{\alpha(t)}{\beta(t)}\right) - \tfrac\pi2,
\end{equation}
and $\alpha$ can be chosen, for example, as 
\begin{equation} \label{eq:alphanew2}
\alpha(t) = \tfrac34 - t, \quad t \in [0,\tfrac12].
\end{equation}
\end{subequations}
Observe that \eqref{eq:xnew_ellipsea} intersects the ellipse $\partial\Omega$ 
at the points $(\pm 2 \alpha(t), \sqrt{1-\alpha^2(t)})^T$ at right angles,
see Figure~\ref{fig:xnewell} for a visualization.
For the convergence experiments with \eqref{eq:xnew_ellipse}, as before, we
compute the residual \eqref{eq:f} of \eqref{eq:xnew_ellipsea} with respect to
\eqref{eq:csf} and modify the right hand sides in the discrete schemes
appropriately. The results for the scheme \eqref{eq:feafilt}
are shown in Table~\ref{tab:fd_filteredellipse2circle}.
Once again we observe the optimal convergence rates proven in
Theorem~\ref{thm:main}.
\begin{figure}
\center
\includegraphics[angle=-90,width=0.4\textwidth]{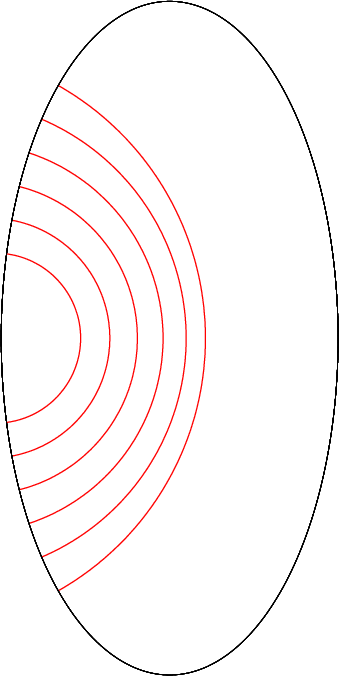}
\caption{The constructed solution \eqref{eq:xnew_ellipse} inside
a $2\!:\!1$ ellipse at times $t=0,0.1,0.2,0.3,0.4,0.5$.
}
\label{fig:xnewell}
\end{figure}%
\begin{table}
\center
\begin{tabular}{|r|c|c|c|c|}
\hline
$J$ & $\errorxL$ & EOC & $\errorxH$ & EOC \\ \hline
32   &1.2504e-02& ---&7.0620e-02& --- \\
64   &3.4277e-03&1.87&3.5087e-02&1.01 \\
128  &8.9026e-04&1.94&1.7503e-02&1.00 \\
256  &2.2631e-04&1.98&8.7463e-03&1.00 \\
512  &5.7013e-05&1.99&4.3725e-03&1.00 \\
1024 &1.4306e-05&1.99&2.1862e-03&1.00 \\
2048 &3.5828e-06&2.00&1.0931e-03&1.00 \\
4096 &8.9650e-07&2.00&5.4654e-04&1.00 \\
\hline
\end{tabular}
\caption{Errors for the convergence test for \eqref{eq:xnew_ellipse}
inside a $2\!:\!1$ ellipse
over the time interval $[0,0.5]$. 
We use $\Delta t = h$ for the scheme \eqref{eq:feafilt}.
We also display the experimental orders of convergence (EOC).
}
\label{tab:fd_filteredellipse2circle}
\end{table}%

\begin{remark} \label{rem:pd}
As a comparison, we recall from Remark~\ref{rem:fea}
the predictor-corrector scheme \eqref{eq:cnfea},
which in \cite{csftime} was shown to be second order in time for closed curves.
Repeating the two previous convergence experiments for this scheme yields
the results reported in Tables~\ref{tab:fd_cn7semicircleg} and
\ref{tab:fd_cn7ellipse2circle}. As we can see, the convergence experiment
for \eqref{eq:xnew_ellipse} shows a suboptimal convergence rate in the case of a
curved boundary.
\begin{table}
\center
\begin{tabular}{|r|c|c|c|c|}
\hline
$J$ & $\errorxL$ & EOC & $\errorxH$ & EOC \\ \hline
32   &3.5231e-03& ---&9.0571e-02& --- \\
64   &1.0026e-03&1.81&4.5289e-02&1.00 \\
128  &2.6707e-04&1.91&2.2645e-02&1.00 \\
256  &6.3741e-05&2.07&1.1323e-02&1.00 \\
512  &1.5567e-05&2.03&5.6613e-03&1.00 \\
1024 &3.9201e-06&1.99&2.8307e-03&1.00 \\
2048 &9.8356e-07&1.99&1.4153e-03&1.00 \\
4096 &2.4516e-07&2.00&7.0766e-04&1.00 \\
\hline
\end{tabular}
\caption{Errors for the convergence test for \eqref{eq:csfsolxg} with
\eqref{eq:g} over the time interval $[0,0.4]$. 
We use $\Delta t = h$ for the scheme \eqref{eq:cnfea}.
We also display the experimental orders of convergence (EOC).
}
\label{tab:fd_cn7semicircleg}
\end{table}%
\begin{table}
\center
\begin{tabular}{|r|c|c|c|c|}
\hline
$J$ & $\errorxL$ & EOC & $\errorxH$ & EOC \\ \hline
32   &7.1813e-03& ---&6.9987e-02& --- \\
64   &2.1706e-03&1.73&3.4988e-02&1.00 \\
128  &6.3153e-04&1.78&1.7492e-02&1.00 \\
256  &1.8295e-04&1.79&8.7452e-03&1.00 \\
512  &5.3702e-05&1.77&4.3724e-03&1.00 \\
1024 &1.6113e-05&1.74&2.1862e-03&1.00 \\
2048 &4.9597e-06&1.70&1.0931e-03&1.00 \\
4096 &1.5666e-06&1.66&5.4654e-04&1.00 \\
\hline
\end{tabular}
\caption{Errors for the convergence test for \eqref{eq:xnew_ellipse} 
inside a $2\!:\!1$ ellipse 
over the time interval $[0,0.5]$. 
We use $\Delta t = h$ for the scheme \eqref{eq:cnfea}.
We also display the experimental orders of convergence (EOC).
}
\label{tab:fd_cn7ellipse2circle}
\end{table}%
\end{remark}

As a further convergence test, we would like to consider the following family
of curves evolving within the unit ball $\Omega = \bB_1^3(0) \subset \bR^3$.
Let $\alpha : [0,T] \to (0,1)$ as before and set
\begin{subequations} \label{eq:xsol_klaus3d}
\begin{equation} \label{eq:xsol_klaus3da}
x(\rho,t)=(0,0,\alpha^{-1}(t))^T +  \sqrt{\alpha^{-2}(t) - 1} 
(\cos t \sin g_3(\rho,t),\sin t \sin g_3(\rho,t),\cos g_3(\rho,t))^T,
\end{equation}
where 
\begin{equation} \label{eq:g3}
g_3(\rho,t)=(2 \rho -1) \arcsin \alpha(t) + \pi.
\end{equation}
\end{subequations}
Here $\alpha$ may, for example, be defined as in \eqref{eq:alphanew2}.
Observe that \eqref{eq:xsol_klaus3d} parameterizes the arc of a circle of 
radius $\sqrt{\alpha^{-2}(t) - 1}$ around the centre $(0,0,\alpha^{-1}(t))^T$,
lying within a vertical hyperplane that is obtained from rotating
$\bR\times \{0\} \times \bR$ around the $z$-axis by an angle $t$. 
The arc meets the unit sphere orthogonally at the two points
$( \pm \sqrt{1-\alpha(t)^2} \cos t, \pm \sqrt{1-\alpha(t)^2} \sin t,
\alpha(t))^T$,
see Figure~\ref{fig:klaus3d} for a visualization.
For the convergence experiment with \eqref{eq:xsol_klaus3d}, as before, we
compute the residual \eqref{eq:f} of \eqref{eq:xsol_klaus3da} with respect to
\eqref{eq:csf} and modify the right hand sides in the discrete schemes
appropriately. The results for the scheme \eqref{eq:feafilt}
are shown in Table~\ref{tab:fd_filteredklaus3d}.
Once again we observe the optimal convergence rates proven in
Theorem~\ref{thm:main}.
\begin{figure}

\center
\includegraphics[angle=-0,width=0.4\textwidth]{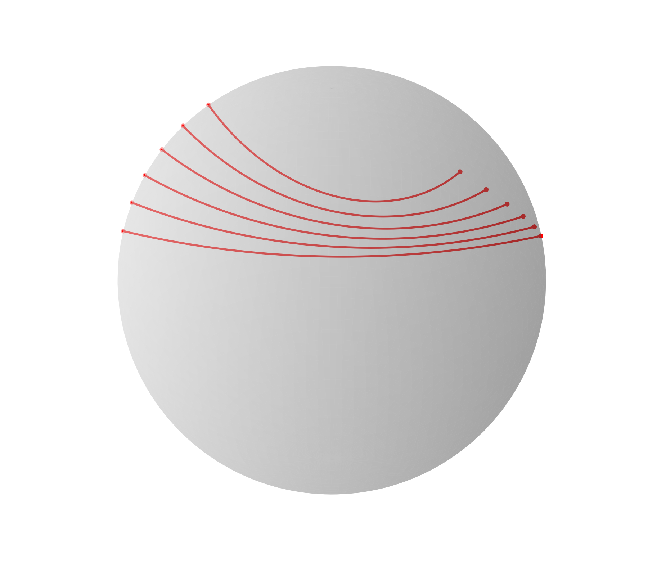}
\caption{The constructed solution \eqref{eq:xsol_klaus3d} inside
the unit sphere $\bB_1^3(0)$ at times $t=0,0.1,0.2,0.3,0.4,0.5$.
}
\label{fig:klaus3d}
\end{figure}%
\begin{table}
\center
\begin{tabular}{|r|c|c|c|c|}
\hline
$J$ & $\errorxL$ & EOC & $\errorxH$ & EOC \\ \hline
32   &3.6748e-03& ---&2.2888e-02& --- \\
64   &9.6355e-04&1.93&1.1444e-02&1.00 \\
128  &2.4513e-04&1.97&5.7219e-03&1.00 \\
256  &6.1721e-05&1.99&2.8610e-03&1.00 \\
512  &1.5480e-05&2.00&1.4305e-03&1.00 \\
1024 &3.8757e-06&2.00&7.1524e-04&1.00 \\
2048 &9.6962e-07&2.00&3.5762e-04&1.00 \\
4096 &2.4249e-07&2.00&1.7881e-04&1.00 \\
\hline
\end{tabular}
\caption{Errors for the convergence test for \eqref{eq:xsol_klaus3d}
inside the unit sphere $\bB_1^3(0)$ over the time interval $[0,0.5]$. 
We use $\Delta t = h$ for the scheme \eqref{eq:feafilt}.
We also display the experimental orders of convergence (EOC).
}
\label{tab:fd_filteredklaus3d}
\end{table}%

For the next numerical simulation we consider the evolution of a curve
inside the elliptic domain \eqref{eq:Omega_ell}. 
We start from a horizontal line, at height $y = -0.1$, which means that this
particular initial data does not satisfy the right contact angle condition
\eqref{eq:bc}. Of course, for the numerical scheme that is not a problem.
We show the
evolution for the discrete parameters $J=256$ and $\Delta t = 10^{-4}$
in Figure~\ref{fig:ell21}. As is to be expected,
the curve shrinks to a point.
Observe that the ratio \eqref{eq:ratio} remains bounded and
decreases towards 1 as the curve shrinks.
\begin{figure}
\center
\includegraphics[angle=-90,width=0.4\textwidth]{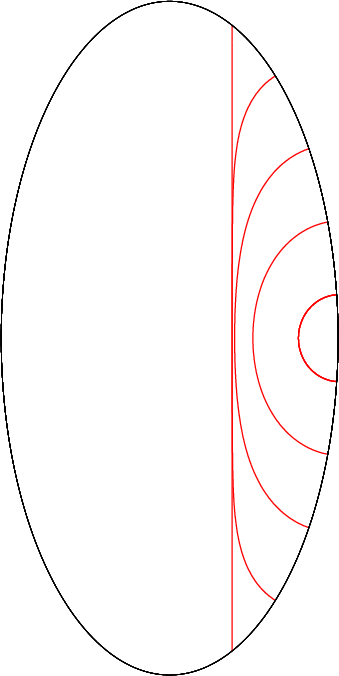}
\qquad
\includegraphics[angle=-90,width=0.31\textwidth]{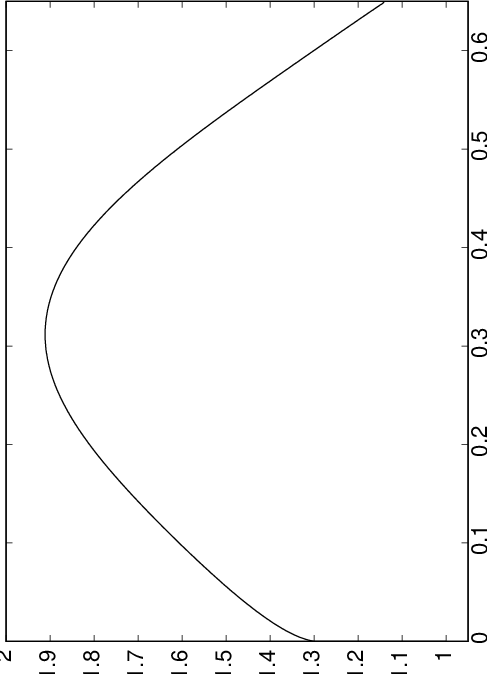}\\
\caption{Curve shortening flow inside a $2\!:\!1$ ellipse. We show the discrete
solution at times $t=0,0.1,0.3,0.5,0.65$.
On the right we show the evolution of $\ratio^m$ over time.
}
\label{fig:ell21}
\end{figure}%

Next we start with a horizontal line of length about $\frac54$, and
at height $y = 0.01$, inside the domain $\Omega = \bB_1^2(0) \setminus
\overline{\bB_\frac14^2(\binom{-\frac12}{0})}$. This experiment is inspired by
Figure~3 in \cite{DeckelnickE98}, where a very similar setup was considered. 
We show a simulation with $J= 256$ and $\Delta t = 10^{-4}$
in Figure~\ref{fig:holedisk}. We can see that 
because the initial data starts
just above the stationary solution represented by the straight line from
$(-\frac14,0)$ to $(1,0)$, the curve slowly travels around the annular domain
to finally settle on the global minimizer: the line segment from
$(-1,0)$ to $(-\frac34,0)$. It is noteworthy that the polygonal curve remains
nearly equidistributed throughout the evolution, and eventually 
assumes an equidistributed numerical steady state.
\begin{figure}
\center
\includegraphics[angle=-90,width=0.3\textwidth]{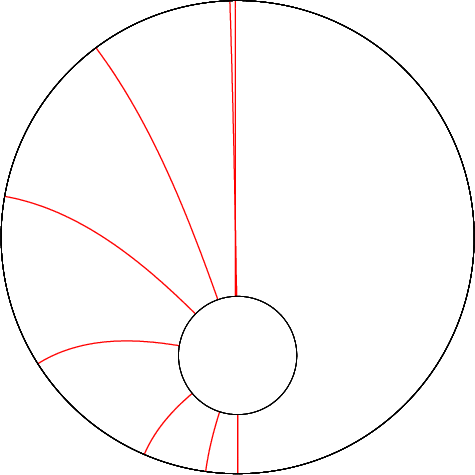} \qquad
\includegraphics[angle=-90,width=0.31\textwidth]{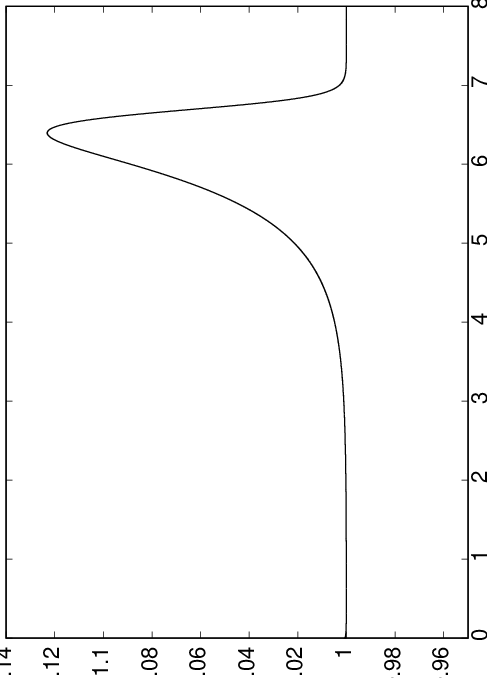}\\
\caption{Curve shortening flow inside a disk with a hole. 
We show the discrete solution at times $t=0,1,5,6,6.5,6.8,7,8$.
On the right we show the evolution of $\ratio^m$ over time.
}
\label{fig:holedisk}
\end{figure}%

In our next experiment we consider an open helix in $\bR^3$ evolving inside
$\Omega = (-\frac14,\frac54) \times \bR^2$.
Here the helix part of the initial curve is defined by
\begin{equation}
 x_0(\varrho) = (\varrho, \sin(8\,\pi\varrho), \cos(8\,\pi\,\varrho))^T
\,,\quad \varrho \in [0,1]\,,
\label{eq:helix}
\end{equation}
and the initial curve is constructed from \eqref{eq:helix} by merging it with
the two line segments \linebreak $[(-\frac14,0,1)^T, (0,0,1)^T]$ and 
$[(1,0,1)^T, (\frac54,0,1)^T]$. 
A simulation for $J=512$ and $\Delta t = 10^{-4}$ is shown
in Figure~\ref{fig:openhelix}, where we notice that the helix straightens to a
straight line. Of course, while it does so, the two endpoints slide
orthogonally along the two hyperplanes that make up $\partial\Omega$.
\begin{figure}
\center
\mbox{
\includegraphics[angle=-0,width=0.25\textwidth]{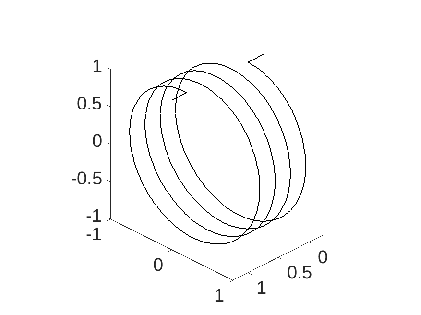}
\includegraphics[angle=-0,width=0.25\textwidth]{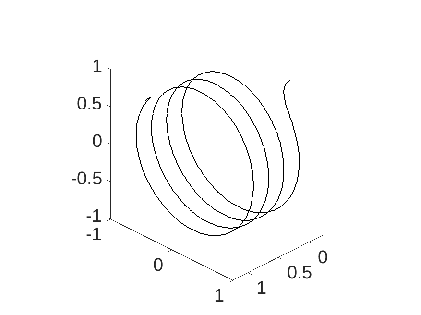}
\includegraphics[angle=-0,width=0.25\textwidth]{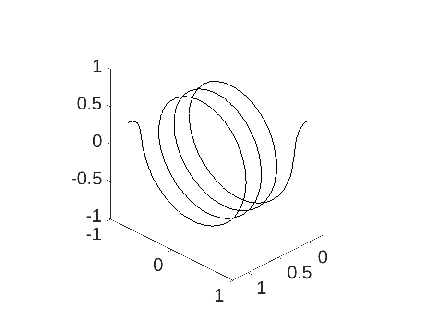}
\includegraphics[angle=-0,width=0.25\textwidth]{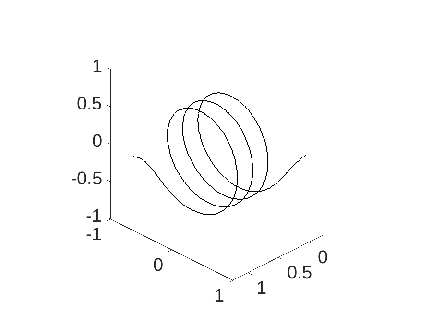}
}
\mbox{
\includegraphics[angle=-0,width=0.25\textwidth]{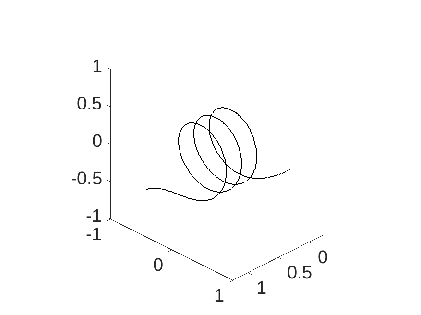}
\includegraphics[angle=-0,width=0.25\textwidth]{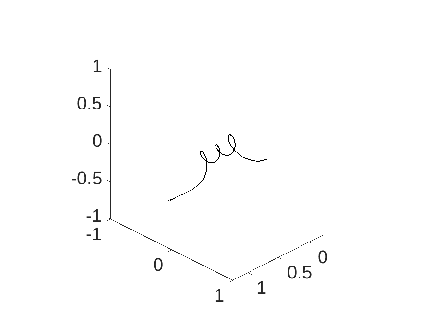}
\includegraphics[angle=-0,width=0.25\textwidth]{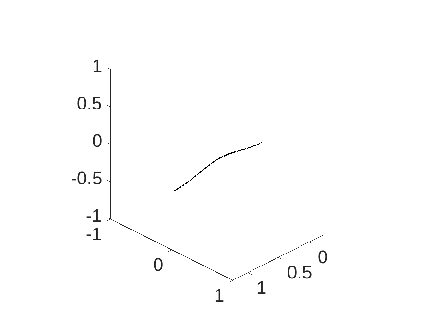}
\includegraphics[angle=-0,width=0.25\textwidth]{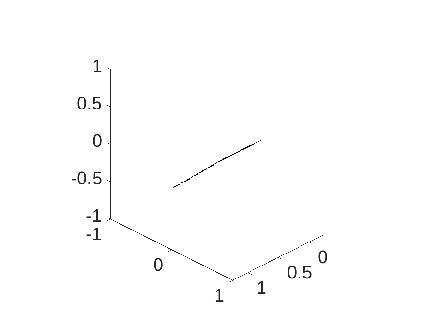}
}
\caption{Curve shortening flow in $\bR^3$.
We show the discrete solution at times $t=0,0.1,0.2,0.3,0.4,0.5,0.6,0.7$.
}
\label{fig:openhelix}
\end{figure}%

\subsection{Closed curves}

For the sake of completeness, we also consider some numerical simulations for
closed curves. Observe that in this case, the scheme \eqref{eq:cnfea} from
\cite{csftime} offers an alternative second order in time method. Here the
advantage of the scheme \eqref{eq:feafilt} is that only a single linear system
needs to be solved at each time step, rather than two.

We begin with a convergence experiment consisting of radially shrinking
circles, that is
\begin{equation} \label{eq:closedsolxg}
x(\rho,t) = (1 - 2 t)^{\frac12} \left( \sin g_1(2\rho), \cos g_1(2\rho) \right)^T
\quad \rho \in I,\ t \in [0, \tfrac12),
\end{equation}
with \eqref{eq:g}. 
The results are shown in Table~\ref{tab:fd_filteredcircleg},
confirming the estimates proven in Theorem~\ref{thm:main} in the case of
closed curves.
\begin{table}
\center
\begin{tabular}{|r|c|c|c|c|}
\hline
$J$ & $\errorxL$ & EOC & $\errorxH$ & EOC \\ \hline
32   &8.9306e-03& ---&3.6212e-01& --- \\
64   &2.4955e-03&1.84&1.8114e-01&1.00 \\
128  &6.5729e-04&1.92&9.0578e-02&1.00 \\
256  &1.5890e-04&2.05&4.5290e-02&1.00 \\
512  &3.9051e-05&2.02&2.2645e-02&1.00 \\
1024 &9.8170e-06&1.99&1.1323e-02&1.00 \\
2048 &2.4610e-06&2.00&5.6613e-03&1.00 \\
4096 &6.1389e-07&2.00&2.8307e-03&1.00 \\
\hline
\end{tabular}
\caption{Errors for the convergence test for \eqref{eq:closedsolxg} with
\eqref{eq:g} over the time interval $[0,0.4]$. 
We use $\Delta t = h$ for the scheme \eqref{eq:feafilt}.
We also display the experimental orders of convergence (EOC).
}
\label{tab:fd_filteredcircleg}
\end{table}%

In our final experiment we consider
a closed helix in $\bR^3$, similarly to \cite[Figure~2]{curves3d}. 
Here the initial curve is constructed from \eqref{eq:helix}
by connecting $x_0(0)$ and $x_0(1)$ with a polygon that visits the origin
and $(1,0,0)^T$. A simulation for $J=512$ and $\Delta t = 10^{-4}$ is shown 
in Figure~\ref{fig:helix}. We observe that the helix attempts to unravel while
it shrinks, and it eventually shrinks to a point.
\begin{figure}
\center
\includegraphics[angle=-0,width=0.3\textwidth]{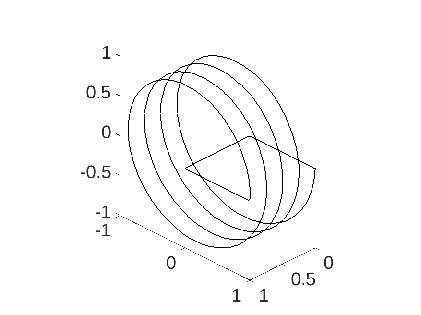}
\includegraphics[angle=-0,width=0.3\textwidth]{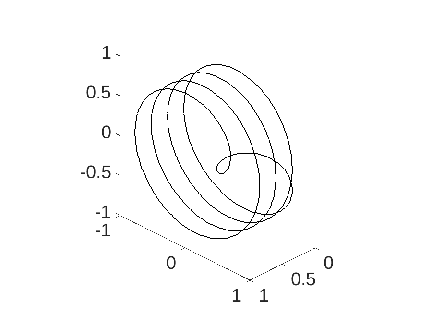}
\includegraphics[angle=-0,width=0.3\textwidth]{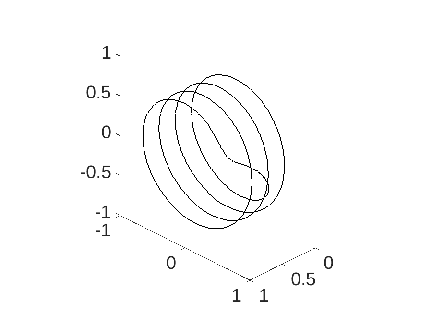}
\includegraphics[angle=-0,width=0.3\textwidth]{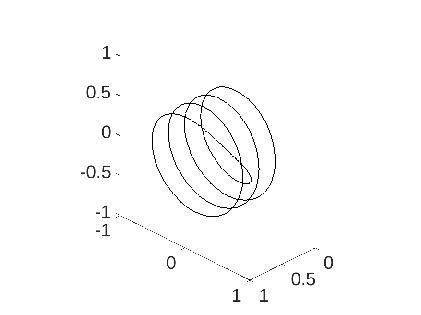}
\includegraphics[angle=-0,width=0.3\textwidth]{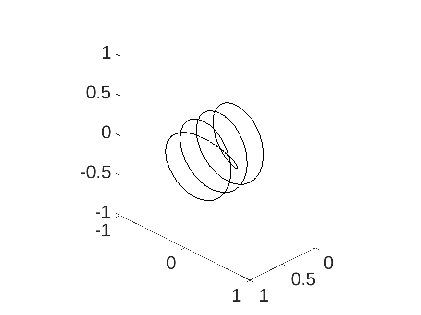}
\includegraphics[angle=-0,width=0.3\textwidth]{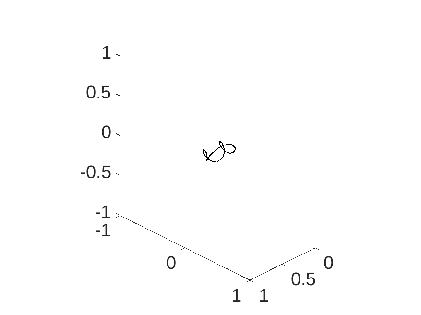}
\caption{Curve shortening flow in $\bR^3$.
We show the discrete solution at times $t=0,0.1,0.2,0.3,0.4,0.5$.
}
\label{fig:helix}
\end{figure}%

\begin{appendix}
\setcounter{equation}{0} 
\renewcommand{\theequation}{\Alph{section}.\arabic{equation}}
\section{Appendix} \label{sec:AppA}

In this appendix we propose the solution of a single linear system in order to
obtain discrete initial data satisfying
\eqref{eq:este1}. 
Set $q^0_{j}:=| \delta^- x^{0}_{j}|$, $j=1,\ldots,J$, and recall the definition
\eqref{eq:defPA}.
Let $x^1_h \in \Vh$ be the solution of the problem
\begin{subequations} \label{eq:DE98prednew}
\begin{align}
& \tfrac{1}{2} \bigl( (q^0_{j})^2 +  (q^0_{j+1})^2 \bigr) \frac{x^1_{h,j} - x^0_{j}}{\Delta t} - \delta^2 x^{1}_{h,j} =0, \quad
j=1,\ldots,J-1, \label{eq:DE98prednewa} \\
&\frac{x^{1}_{h,0} - x^0_{0}}{\Delta t} \cdot \left\{ \nabla F(x^0_{0}) 
+ \frac{2\Delta t }{h  (q^0_{1})^2} A(x^0_0) \delta^+ x^0_{0} 
\right\} =0,
\label{eq:DE98prednewb} \\
&\frac{x^{1}_{h,J} - x^0_{J}}{\Delta t} \cdot \left\{  \nabla F(x^0_{J}) 
- \frac{2\Delta t }{ h  (q^0_{J})^2}   A(x^0_J) \delta^- x^0_{J}  \right\} =0,
\label{eq:DE98prednewc} \\
&  (q^0_{1})^2 \Bigl(\Id
+  \frac{2 \Delta t }{h (q^0_{1})^2} (\delta^+ x^0_{0}\cdot 
\nabla F (x^0_{0}))  A(x^0_0) \Bigr) P(x^0_0)
\frac{x^{1}_{h,0} - x^0_{0}}{\Delta t}
- \frac{2}{h} P(x^0_0) \delta^+ x^{1}_{h,0}  = 0,
\label{eq:DE98prednewd} \\
&   (q^0_{J})^2 \Bigl( \Id
- \frac{2 \Delta t }{h(q^0_{J})^2} (\delta^- x^0_{J}\cdot 
\nabla F (x^0_{J})) A(x^0_J)   \Bigr) P(x^0_J)
 \frac{x^{1}_{h,J} - x^0_{J}}{\Delta t} 
 + \frac{2}{h}  P(x^0_J) \delta^- x^{1}_{h,J} = 0,
\label{eq:DE98prednewe} 
\end{align}
\end{subequations}
where $P(z), A(z)$ for $z \in \partial \Omega$ are defined in \eqref{eq:defPA}.

\begin{lemma} \label{lem:errinit}
Suppose that \eqref{eq:pde} has a smooth solution satisfying \eqref{eq:regul}. 
Then there exist $0 < \gamma \leq 1$ and  $h_0>0$ such that  \eqref{eq:DE98prednew} has a unique solution $x^1_h \in \Vh$ satisfying
\begin{equation} \label{eq:212}
\| I_h x(\cdot,t_1) - x^1_h \|_{1}^2 \leq c(h^4 + (\Delta t)^4),
\end{equation}
provided that $0<\Delta t \leq \gamma h$ and $0<h \leq h_0$. 
\end{lemma}
\begin{proof}
Similarly to the proof of Lemma~\ref{lem:stab}, we can prove the well-posedness
of \eqref{eq:DE98prednew} via the uniqueness of the solution to the following
homogeneous system. Let $\bar x_h \in \Vh$ be such that
\begin{subequations} \label{eq:app:DE98homo}
\begin{align}
& \frac{1}{2 \Delta t} \bigl( (q^0_{j})^2 +  (q^0_{j+1})^2 \bigr) \bar x_{h,j} - \delta^2 \bar x_{h,j} =0, \quad
j=1,\ldots,J-1, \label{eq:app:DE98homoa} \\
&\bar x_{h,0}  \cdot \left\{ \nabla F(x^0_{0}) 
+ \frac{2\Delta t }{ h  (q^0_{1})^2}   A(x^0_0) \delta^+ x^0_0 
\right\} =0, \label{eq:app:DE98homob} \\
&\bar x_{h,J}  \cdot \left\{  \nabla F(x^0_{J}) 
- \frac{2\Delta t }{ h  (q^0_{J})^2}   A(x^0_J) \delta^- x^0_{J}  \right\} =0,
\label{eq:app:DE98homoc} \\
&  (q^0_{1})^2 \Bigl(\Id
+  \frac{2 \Delta t }{h (q^0_{1})^2} (\delta^+ x^0_{0}\cdot 
\nabla F (x^0_{0}))  A(x^0_0) \Bigr) P(x^0_0)  \frac{\bar x_{h,0}}{\Delta t}
- \frac{2}{h} P(x^0_0) \delta^+ \bar x_{h,0}  = 0,
\label{eq:app:DE98homod} \\
&   (q^0_{J})^2 \Bigl( \Id
- \frac{2 \Delta t }{h(q^0_{J})^2} (\delta^- x^0_{J}\cdot 
\nabla F (x^0_{J})) A(x^0_J)   \Bigr) P(x^0_J) 
 \frac{\bar x_{h,J}}{\Delta t} 
 + \frac{2}{h}  P(x^0_J) \delta^- \bar x_{h,J} = 0.
\label{eq:app:DE98homoe}
\end{align}
\end{subequations}
Multiplying \eqref{eq:app:DE98homoa} by $ h \bar x_{h,j}$, 
summing over $j=1,\ldots,J-1$, and using \eqref{eq:sbp} yields that
\begin{align} \label{eq:app:homo1}
& \frac{h}{2 \Delta t} \sum_{j=1}^{J-1} \bigl( (q^0_{j})^2 +  (q^0_{j+1})^2 \bigr)| \bar x_{h,j}|^2
+ | \bar x_h |_{1,h}^2
= \bar x_{h,J} \cdot \delta^- \bar x_{h,J} -  \bar x_{h,0} \cdot \delta^+ \bar x_{h,0} .
\end{align}
Abbreviating $d_J:= \nabla F(x^0_J)$ and $P_J:= P(x^0_J)$  we may write
\begin{equation} \label{eq:app:decompJ}
\bar x_{h,J} \cdot \delta^- \bar x_{h,J} = \bigl( \bar x_{h,J} \cdot d_J \bigr) \bigl( \delta^- \bar x_{h,J} \cdot  d_J \bigr) + P_J \bar x_{h,J} \cdot P_J \delta^- \bar x_{h,J} .
\end{equation}
From the boundary conditions \eqref{eq:app:DE98homoc} and \eqref{eq:app:DE98homoe}, we can substitute:
\begin{align}
\bar x_{h,J} \cdot d_J & = \frac{2\Delta t}{ h  (q^0_{J})^2} \bigl(  \bar x_{h,J} \cdot A(x^0_J)  \delta^- x^0_{J} \bigr) = \frac{2\Delta t}{ h (q^0_{J})^2} \bigl( P_J \bar x_{h,J} \cdot D^2 F(x^0_J) P_J \delta^- x^0_J \bigr)   , \label{eq:bound1} \\
P_J \delta^- \bar x_{h,J} & =  - \frac{h}{2\Delta t} (q^0_J)^2
\Bigl( \Id
- \frac{2 \Delta t }{h(q^0_{J})^2} (\delta^- x^0_{J}\cdot 
\nabla F (x^0_{J})) A(x^0_J)   \Bigr) P_J
\bar x_{h,J} . \label{eq:bound2}
\end{align}
Note that \eqref{eq:pdec} and $\delta^- x^0_{J}= x^0_{\rho}(\rho_J)+\mathcal O(h)$ imply that $| P_J \delta^- x^0_{J} | \leq c h$, so that
\[
| \bar x_{h,J} \cdot d_J | \leq c \Delta t | P_J \bar x_{h,J} | \leq c h | P_J \bar x_{h,J} | ,
\]
since $\Delta t \leq \gamma h \leq h$. Furthermore, using again that $\Delta t \leq \gamma h$, we have
\begin{align} 
\Bigl( \Id
- \frac{2 \Delta t }{h(q^0_{J})^2} (\delta^- x^0_{J}\cdot 
\nabla F (x^0_{J})) A(x^0_J)   \Bigr) v \cdot v
 & \geq \bigl( 1 - \frac{2 \Delta t q^0_J | A(x^0_J)| }{h (q^0_J)^2}  \bigr) | v |^2 \nonumber \\
 & \geq \bigl( 1 - 2 \gamma \frac{ | A(x^0_J)| }{q^0_J}  \bigr) | v |^2  \geq  \tfrac{1}{2} | v |^2 \label{eq:lowbound}
\end{align}
for all $v \in \bR^n$, provided that $\gamma | A(x^0_J) | \leq \tfrac14 q^0_J$.  Inserting \eqref{eq:bound1}, \eqref{eq:bound2} into \eqref{eq:app:decompJ}, and using the above estimates
as well as \eqref{eq:inv1}, we derive
\begin{align*} 
\bar x_{h,J} \cdot \delta^- \bar x_{h,J} 
& \leq c  h | P_J  \bar x_{h,J} | \, |  \delta^- \bar x_{h,J} |   - \frac{h}{4\Delta t} (q^0_J)^2  | P_J  \bar x_{h,J}|^2 \\
& \leq - \frac{h}{8\Delta t} (q^0_J)^2  | P_J  \bar x_{h,J} |^2 + c h \Delta t |  \delta^- \bar x_{h,J} |^2 \leq 
- \frac{h}{8\Delta t} (q^0_J)^2  | P_J  \bar x_{h,J} |^2 + c \Delta t | \bar x_h |_{1,h}^2.
\end{align*}
Arguing in the same way  for the left boundary point, we hence deduce from \eqref{eq:app:homo1} and \eqref{eq:app:decompJ} that
\[
\frac{h}{8\Delta t} \Bigl( (q^0_{1})^2  | P_0 \bar x_{h,0} |^2 +  (q^0_J)^2  | P_J  \bar x_{h,J} |^2 \Bigr)  + 
\frac{h}{2 \Delta t} \sum_{j=1}^{J-1} \bigl( (q^0_{j})^2 +  (q^0_{j+1})^2 \bigr)| \bar x_{h,j}|^2
+ (1 - c \Delta t)  | \bar x_h |_{1,h}^2 \leq 0.  
\]
Since $c \Delta t \leq c h  \leq c h_0  \leq \frac12$ if $h_0$ is small enough, 
 we deduce with the help of \eqref{eq:bound1} that 
$\bar x_{h,j} = 0$, $j=0,\ldots,J$ and hence there exists a unique solution to the system \eqref{eq:DE98prednew}.

Next we would like to prove the estimate \eqref{eq:212}. To this end, 
we rewrite \eqref{eq:DE98prednewa} in the form
\begin{equation} \label{eq:fd}
x^{1}_{h,j} - x^0_{j} - \Delta t \, \dfrac{2}{  (q^{0}_{j})^2 + (q^{0}_{j+1})^2} \,   \delta^2 x^{1}_{h,j}  = 0, \qquad j=1,\ldots,J-1.
\end{equation}
Combining \eqref{eq:fd} with \eqref{eq:csf}, and using the notation
\eqref{eq:em}, we derive the error relation
\begin{align*}
 & e^{1}_j  -  \Delta t \,  \dfrac{2}{  (q^{0}_{j})^2 + (q^{0}_{j+1})^2} \,   \delta^2 e^{1}_{j} 
 =  \Delta t \left(  x_t(\rho_j,0) - \frac{x^{1}_j-x^0_j}{ \Delta t} \right)+ \frac{\Delta t}{ | x^0_\rho(\rho_j) |^2} \bigl( x^{1}_{\rho \rho}(\rho_j) -x^0_{\rho \rho}(\rho_j) \bigr) \\ & \qquad 
 + \Delta t \left(  \dfrac{2}{  (q^{0}_{j})^2 + (q^{0}_{j+1})^2} - \frac{1}{ | x^0_\rho(\rho_j) |^2} \right) \delta^2 x^{1}_j  + \Delta t \, \frac{1}{ | x^0_\rho(\rho_j) |^2} \bigl( 
\delta^2 x^{1}_j - x^{1}_{\rho \rho}(\rho_j) \bigr) \\
& \quad
=:\sum_{\ell=1}^4 \mathfrak f^{0}_{\ell,j}, \quad
 j=1,\ldots,J-1.
\end{align*}
If we multiply by $-h \, \delta^2 e^{1}_j$, sum over $j=1,\ldots,J-1$ and use \eqref{eq:sbp} as well as \eqref{eq:regul}, we obtain
\begin{align} \label{eq:erra1}
| e^1_h|_{1,h}^2  + \frac{\Delta t}{16 C_0^2} \sum_{j=1}^{J-1}  h \,   | \delta^2 e^{1}_j |^2   
\leq e^1_J \cdot \delta^- e^{1}_J -   e^1_0 \cdot \delta^+ e^{1}_0 - \sum_{\ell=1}^4 \sum_{j=1}^{J-1} h \,  \mathfrak f^{0}_{\ell,j} \cdot \delta^2 e^{1}_j.  
\end{align}
In order to estimate the terms involving $\mathfrak f^0_{\ell,j}$ we follow the arguments in \cite{csftime} but take into account the boundary terms. For the first two terms we derive using \eqref{eq:sbp}
\begin{align*}
 - \sum_{j=1}^{J-1} h \,  \bigl( \mathfrak f^{0}_{1,j} +\mathfrak f^0_{2,j} \bigr)  \cdot \delta^2 e^{1}_j & = \sum_{j=1}^J h \,  \bigl( \delta^- \mathfrak f^0_{1,j} + \delta^- \mathfrak f^0_{2,j} \bigr)  \cdot \delta^- e^{1}_j \nonumber \\
& \qquad -(\mathfrak f^0_{1,J} + \mathfrak f^0_{2,J}) \cdot \delta^- e^{1}_J + (\mathfrak f^0_{1,0} + \mathfrak f^0_{2,0}) \cdot \delta^+ e^{1}_0 \\
& \leq \tfrac{1}{2} \sum_{j=1}^J h \,  | \delta^- e^{1}_j |^2 + c (\Delta t)^4 -(\mathfrak f^0_{1,J} + \mathfrak f^0_{2,J}) \cdot \delta^- e^{1}_J + (\mathfrak f^0_{1,0} + \mathfrak f^0_{2,0}) \cdot \delta^+ e^{1}_0,
\end{align*}
where we argue as for (3.17), (3.18) in \cite{csftime}. Using (3.20), (3.21) in \cite{csftime} we can estimate the remaining two terms as follows
\begin{align*}
& - \sum_{j=1}^{J-1} h \,  \bigl( \mathfrak f^{0}_{3,j} +\mathfrak f^0_{4,j} \bigr)  \cdot \delta^2 e^{1}_j \leq \frac{\Delta t}{16 C_0^2} \sum_{j=1}^{J-1} h \,  | \delta^2 e^{1}_j |^2 + c \Delta t  h^4 .
\end{align*}
In conclusion we obtain
\begin{align} 
& \tfrac12  | e^{1}_h |_{1,h} ^2 \leq  B^0_J \cdot \delta^- e^{1}_J - B^0_0 \cdot \delta^+ e^{1}_0 + c\bigl( h^4 + (\Delta t)^4 \bigr), \label{eq:err2}
\end{align}
where 
\begin{equation} \label{eq:bmj}
B^0_0:= e^{1}_0 - (\mathfrak f^0_{1,0} + \mathfrak f^0_{2,0}), \quad B^0_J:=  e^{1}_J  - (\mathfrak f^0_{1,J} + \mathfrak f^0_{2,J}).
\end{equation}
It remains to examine the boundary terms in \eqref{eq:err2}. Let us 
decompose
\begin{align}
 B^0_J \cdot \delta^- e^{1}_J  \label{eq:b1}  
 = P_J  B^0_J  \cdot P_J \delta^- e^{1}_J  + \bigl( B^0_J  \cdot d_J  \bigr) \bigl( \delta^- e^{1}_J \cdot d_J \bigr) ,
\end{align}
where we abbreviate again $d_J:=\nabla F(x^0_J)$, $P_J:= P(x^0_J)$ and $A_J:=A(x^0_J)$.
In order to handle the first term in \eqref{eq:b1}, we use \eqref{eq:DE98prednewe} and write
\begin{align} \label{eq:bm2} 
P_J \delta^- e^{1}_J & = P_J \delta^- x^{1}_{h,J} 
- P_J \delta^- x^{1}_J  \nonumber \\
& =  - \tfrac12  \frac{h}{\Delta t}   (q^0_{J})^2 \Bigl( \Id
- \frac{2 \Delta t }{h(q^0_{J})^2} (\delta^- x^0_{J}\cdot d_J) A_J   \Bigr) P_J
 (x^{1}_{h,J} - x^0_{J})
 - P_J \delta^- x^{1}_J . 
\end{align}
Let us consider the last term in the above relation. Taylor expansion yields
\begin{align*}
\delta^- x^{1}_J = x^{1}_\rho(\rho_J) - \tfrac12 h x^{1}_{\rho \rho}(\rho_J) + \mathcal O(h^2) ,
\end{align*}
so that we obtain similarly as in \eqref{eq:difproj}, with the help of \eqref{eq:pdec},
\begin{align} \label{eq:dife}
P_J  \delta^- x^{1}_J   & = \bigl(  P_J - P(x^1_J) \bigr) x^{1}_\rho(\rho_J)  
- \tfrac12 h P_J x^{0}_{\rho\rho}(\rho_J)  
+ \mathcal O(h^2 + h\Delta t)  \nonumber \\ 
&   =  \bigl( \nabla F(x^0_J) \cdot x^0_\rho(\rho_J) \bigr) A_J  (x^1_J -x^0_J) - \tfrac12 h P_J x^{0}_{\rho\rho}(\rho_J)  
+ \mathcal O(h^2 + (\Delta t)^2) \nonumber \\
& =  (\delta^- x^0_J \cdot d_J) A_J  (x^{1}_J - x^0_J) 
- \tfrac12 h (q^0_J)^2 P_J  \frac{x^{1}_J - x^0_J}{\Delta t} 
+ \mathcal O(h^2 + (\Delta t)^2) \nonumber \\
& =  - \frac{h}{2\Delta t} (q^0_J)^2    \left\{ \Id - \frac{2 \Delta t}{h (q^0_J)^2} (\delta^- x^0_J \cdot d_J) A_J  \right\} P_J (x^{1}_J - x^0_J)+
\mathcal O(h^2 + (\Delta t)^2).
\end{align} 
If we insert the above relation into \eqref{eq:bm2}, we obtain
\begin{align*}
& P_J \delta^- e^{1}_J 
 = - \frac{h}{2\Delta t}  (q^0_{J})^2  \left\{ \Id - \frac{2 \Delta t}{h (q^0_J)^2} (\delta^- x^0_J \cdot d_J) A_J  \right\}
P_J e^{1}_{J}  
 + \mathcal O(h^2 + (\Delta t)^2).
\end{align*}
On noting that $|\mathfrak f^0_{1,J}| + | \mathfrak f^0_{2,J}| \leq c (\Delta t)^2$, as well as \eqref{eq:regul} and \eqref{eq:lowbound}, we infer that 
\begin{align}
P_J B^0_J \cdot P_J \delta^- e^{1}_J & =  P_J \bigl( e^1_J - (\mathfrak f^0_{1,J} + \mathfrak f^0_{2,J}) \bigr) \cdot  P_J \delta^- e^{1}_J 
\nonumber \\ & 
\leq - \frac{h}{4 \Delta t} (q^0_J)^2  | P_J   e^1_J |^2 +  c \bigl|  P_J e^{1}_{J}  \bigr| \bigl( h^2 + (\Delta t)^2  \bigr) + c (h^4 + (\Delta t)^4)  \nonumber  \\ & 
 \leq \bigl( - \tfrac1{16} c_0^2 + \epsilon \bigr)   \frac{h}{\Delta t} | P_J   e^{1}_{J} |^2 + c_\epsilon \bigl(  h^4 + (\Delta t)^4 \bigr),
 \label{eq:est1} 
\end{align}
since $\Delta t \leq h$. 
Let us next consider the second term in \eqref{eq:b1}. We write
\begin{align}
B^0_J & = x^{1}_{h,J} - x^{1}_J -  \Delta t \left(  x_t(\rho_J,0) - \frac{x^{1}_J-x^0_J}{ \Delta t} \right)- \frac{\Delta t}{ | x^0_\rho(\rho_J) |^2} \bigl( x^{1}_{\rho \rho}(\rho_J) -x^0_{\rho \rho}(\rho_J) \bigr) \nonumber \\
& = x^{1}_{h,J} - x^0_{J}  -   \Delta t \frac{x^{1}_{\rho \rho}(\rho_J) }{ | x^0_\rho(\rho_J) |^2}. \label{eq:bmj1}
\end{align}
Using \eqref{eq:bmj1}, \eqref{eq:DE98prednewc} and the fact that $x^1_{\rho \rho}(\rho_J) \cdot \nabla F(x^1_J)=0$, we obtain
\begin{align}\label{eq:bmj2} 
B^0_J \cdot d_J & = (x^{1}_{h,J} - x^0_{J}) \cdot d_J -  \Delta t   \frac{x^{1}_{\rho \rho}(\rho_J) }{ | x^0_\rho(\rho_J) |^2}   \cdot d_J \nonumber \\ & 
=   \frac{2\Delta t }{ h (q^0_{J})^2}  A_J  \delta^- x^0_{J} \cdot (x^{1}_{h,J} - x^0_{J})
 +  \Delta t
\frac{x^{0}_{\rho \rho}(\rho_J) }{ | x^0_\rho(\rho_J) |^2}   \cdot \bigl(  \nabla F(x^{1}_J)  -  \nabla F(x^0_J)\bigr)  \nonumber \\
& \quad + \Delta t \frac{x^{1}_{\rho \rho}(\rho_J) - x^0_{\rho \rho}(\rho_J) }{ | x^0_\rho(\rho_J) |^2}   \cdot \bigl( \nabla F(x^{1}_J)  -  \nabla F(x^0_J)  \bigr) =: B_{1,1}+B_{1,2}+B_{1,3}.
\end{align}
Let us focus first on $B_{1,2}$. On recalling \eqref{eq:difnormal}, we have
\begin{align*}
& B_{1,2}   
=  \Delta t A_J  (x^{1}_J - x^0_J) \cdot  \frac{x^{0}_{\rho \rho}(\rho_J) }{ | x^0_\rho(\rho_J) |^2}  + \mathcal O((\Delta t)^3).
\end{align*} 
Taylor expansion yields
\begin{align*}
\frac{x^0_{\rho \rho}(\rho_J)}{| x^0_{\rho}(\rho_J)|^2} =  \frac{2}{h (q^0_J)^2}  \bigl( x^0_\rho (\rho_J) -\delta^- x^0_J \bigr) + \mathcal O(h) ,
\end{align*}
so that, on recalling that $P(x^0_J) x^0_\rho(\rho_J)=0$, we
have
\begin{align}
& B_{1,2} = -  \frac{2 \Delta t }{h (q^0_J)^2}  A_J  \delta^- x^0_J \cdot (x^{1}_J - x^0_J) + \mathcal O(h^3 + (\Delta t)^3). \label{eq:b13}
\end{align}
Inserting \eqref{eq:b13} into \eqref{eq:bmj2}, we derive
\begin{align*}
& B^0_J  \cdot d_J = 
\frac{2\Delta t }{ h (q^0_{J})^2}   A_J \delta^- x^0_{J} \cdot P_J e^{1}_{J} + B_{1,3} + \mathcal O(h^3 + (\Delta t)^3).
\end{align*}
As $| P_J \delta^- x^0_J | \leq ch$ and $\Delta t \leq h$, we therefore deduce that
\begin{align}  \label{eq:bm1est}
& \left| B^0_J  \cdot d_J \right| 
  \leq  c h \left| P_J e^{1}_{J}  \right| +  c(h^3 + (\Delta t)^3 ). 
\end{align}
Hence, using \eqref{eq:bm1est} together with the inverse estimate
\eqref{eq:inv1} yields that
\begin{align}
\bigl(  B^0_J \cdot d_J \bigr)  \bigl( \delta^- e^{1}_J \cdot d_J  \bigr) & \leq  \left| B^0_J  \cdot d_J \right| \, | \delta^- e^{1}_J | 
\leq c h^{\frac{1}{2}}  \left|  P_J e^{1}_{J}    \right| \, | e^1_h |_{1,h}
+ ch^{-\frac{1}{2}} (h^3 + (\Delta t)^3 ) | e^{1}_h |_{1,h} \nonumber  \\ 
& \leq  \frac{c_0^2}{32}   \frac{h}{\Delta t} | P_J   e^{1}_{J}  |^2 +  \bigl(c \Delta t + \tfrac{1}{4} \bigr)  | e^{1}_h |_{1,h}^2 +  c h^{-1}  \bigl( h^6 + (\Delta t)^6 \bigr) \nonumber \\
& \leq  \frac{c_0^2}{32}   \frac{h}{\Delta t} | P_J   e^{1}_{J}  |^2 +  \bigl(c \Delta t + \tfrac{1}{4} \bigr)  | e^{1}_h |_{1,h}^2 +  c  \bigl( h^4 + (\Delta t)^4 \bigr) ,
\label{eq:est2}
\end{align}
where we have used again that $\Delta t \leq h$. 
Inserting \eqref{eq:est1} and \eqref{eq:est2} into \eqref{eq:err2}, and arguing in the same way for the left boundary point, we obtain, after choosing $\epsilon$ sufficiently small,
\begin{align} \label{eq:app:erra3} 
| e^1_h |_{1,h}^2 +  \frac{h}{\Delta t} | P_0   e^{1}_{0}  |^2 + \frac{h}{\Delta t} | P_J e^{1}_{J}  |^2  \leq c \bigl( h^4 + (\Delta t)^4 \bigr).
\end{align}
It follows from the definition of $B^0_J$ and \eqref{eq:bm1est} that
\begin{align*}
& |  e^{1}_{J}  \cdot d_J  | \leq |   B^0_J  \cdot d_J  | + c (\Delta t)^2 \leq c h | P_J  e^{1}_{J} | + c \bigl( h^3 + (\Delta t)^2 \bigr),
\end{align*}
which, combined with an analogous estimate for the left end point, and inserting into \eqref{eq:app:erra3}, yields
\begin{align*}
| e^1_h |_{1,h}^2 + | e^1_0 |^2 + | e^1_J |^2 \leq c \bigl( h^4 + (\Delta t)^4 \bigr).
\end{align*}
Together with \eqref{eq:dpi} and the fact that $\| I_h x(\cdot,t_1) - x^1_h \|_1 \leq c \| e^{1}_h \|_{1,h}$ this implies the assertion of the lemma. 
\end{proof}

\end{appendix}

\def\soft#1{\leavevmode\setbox0=\hbox{h}\dimen7=\ht0\advance \dimen7
  by-1ex\relax\if t#1\relax\rlap{\raise.6\dimen7
  \hbox{\kern.3ex\char'47}}#1\relax\else\if T#1\relax
  \rlap{\raise.5\dimen7\hbox{\kern1.3ex\char'47}}#1\relax \else\if
  d#1\relax\rlap{\raise.5\dimen7\hbox{\kern.9ex \char'47}}#1\relax\else\if
  D#1\relax\rlap{\raise.5\dimen7 \hbox{\kern1.4ex\char'47}}#1\relax\else\if
  l#1\relax \rlap{\raise.5\dimen7\hbox{\kern.4ex\char'47}}#1\relax \else\if
  L#1\relax\rlap{\raise.5\dimen7\hbox{\kern.7ex
  \char'47}}#1\relax\else\message{accent \string\soft \space #1 not
  defined!}#1\relax\fi\fi\fi\fi\fi\fi}

\end{document}